\documentclass[11pt]{article}

\usepackage[margin=1in]{geometry}
\usepackage{amsmath,amssymb,amsthm,mathtools,bm}
\usepackage{enumitem}
\usepackage{booktabs}
\usepackage{graphicx}
\usepackage[percent]{overpic}
\usepackage[font=small,labelfont=bf,labelsep=space]{caption}
\usepackage{xcolor}
\usepackage{placeins}
\usepackage[hidelinks]{hyperref}
\usepackage{microtype}

\newtheorem{theorem}{Theorem}[section]
\newtheorem{proposition}[theorem]{Proposition}
\newtheorem{lemma}[theorem]{Lemma}
\newtheorem{corollary}[theorem]{Corollary}
\theoremstyle{definition}
\newtheorem{definition}[theorem]{Definition}
\newtheorem{remark}[theorem]{Remark}
\newtheorem{example}[theorem]{Example}

\newcommand{\R}{\mathbb{R}}
\newcommand{\Nset}{\mathcal{N}}
\newcommand{\E}{\mathcal{E}}
\newcommand{\one}{\mathbf{1}}
\newcommand{\diag}{\operatorname{diag}}
\newcommand{\Comp}{\operatorname{Comp}}

\title{Exact Fragmentation and Terminal-Cluster Selection in\\
One-Dimensional Finite-Range Normalized Alignment}
\author{Jiangning Chen\\[0.4em]
\small Independent Researcher\\
\small Lexington, MA, USA\\
\small Corresponding author: \href{mailto:chjn1990531@gmail.com}
{\texttt{chjn1990531@gmail.com}}}
\date{}

\hypersetup{
    pdftitle={Exact Fragmentation and Terminal-Cluster Selection in One-Dimensional Finite-Range Normalized Alignment},
    pdfauthor={Jiangning Chen}
}

\begin{document}
\maketitle

\begin{abstract}
Finite-range alignment can end either in a single flock or in several
noninteracting clusters, yet convergence results rarely determine which
outcome follows from a given finite-particle state. We study a
one-dimensional normalized alignment model with a hard interaction cut-off
and obtain exact predictions in an expansive regime. Velocity order is
invariant, so pair separations are nondecreasing and the communication graph
evolves through finitely many irreversible edge deletions. On each fixed
graph, the group inverse of the random-walk Laplacian gives the total relative
displacement remaining before relaxation. Comparing this displacement with
the available interaction slack selects the next deletion and yields a finite
recursion for the complete switching sequence, terminal partition, limiting
cluster velocities, and internal geometry. For path configurations, an
explicit Green kernel gives a necessary-and-sufficient fragmentation
criterion and a sharp critical alignment rate, including asymptotic boundary
contact at criticality. A spectral-geometric condition extends the theory to
an open set of initially nonordered velocities, while a common self-weight
extension covers both self-excluding and self-including local averages.
Numerical computations reproduce the thresholds and multi-event cascades.
The results provide an exact finite-size theory of fragmentation and terminal
state selection for a class of finite-range interacting particle systems.
\end{abstract}



\section{Introduction}

The emergence of coherent motion from local interactions is a basic
collective phenomenon in systems of self-propelled particles
\cite{Reynolds1987,Vicsek1995,CuckerSmale2007}. With long-range influence,
the main question is often whether all particles approach one flock. A hard
interaction cut-off introduces a competing outcome: particles can move out
of range before their velocities equilibrate, permanently separating the
population into several clusters. The final state is then selected by a
feedback loop between relaxation and the loss of interactions.

Normalized alignment makes this selection problem particularly subtle. The
Motsch--Tadmor normalization compensates for variations in local density, but
the resulting interaction matrix is generally nonsymmetric and ordinary
momentum need not be conserved \cite{MotschTadmor2011,MotschTadmor2014}.
Existing results establish flocking under cut-off or short-range interactions
and describe multicluster asymptotics
\cite{Jin2018,MoralesPeszekTadmor2019,HaJinZhang2025}. Such convergence
theorems do not usually identify, for a prescribed finite configuration,
which links disappear, whether fragmentation occurs, or which particles form
each terminal cluster. The purpose of this paper is to obtain that finite-size
information exactly in a nontrivial invariant regime.

We consider the one-dimensional self-excluding normalized average introduced
numerically in \cite{ChenZhangLim2013}. It is closely related to the cut-off
Motsch--Tadmor model studied in \cite{Jin2018}, but the questions and
conclusions are different. Rather than derive a sufficient condition for a
single flock, we determine the complete sequence of interaction losses and
the resulting terminal clusters whenever the initial state is expansive. In
particular, for path configurations we obtain a necessary-and-sufficient
fragmentation criterion and a sharp critical alignment rate. The equality
case describes asymptotic contact with the interaction boundary and is
distinct from finite-time fragmentation.

The mechanism combines geometry with integrated relaxation. When positions
and velocities have the same order, the velocity order is forward invariant.
All pair separations are therefore nondecreasing, so an interaction can be
lost but cannot be created or recovered. On each interval with fixed
interactions, the group inverse of the random-walk Laplacian gives the total
relative displacement remaining before velocity relaxation. Comparing this
quantity with the unused interaction range decides whether the current graph
is terminal. If it is not, the first predicted boundary hit is the next
actual deletion. Repeating the calculation yields the terminal partition,
the limiting velocity of every cluster, and its limiting internal geometry
after at most $|E(0)|$ events.

For an initial path, the adjacent velocity differences evolve through a
symmetric tridiagonal matrix with an explicit positive Green kernel. The
resulting formula reveals a nonlocal finite-particle effect: an initial
velocity gradient across one edge contributes to the accumulated expansion
of every path edge. We further prove that a spectral-geometric condition
allows an open set of initially nonordered velocities to enter the expansive
regime safely. A uniform self-weight extension covers both self-excluding and
self-including local averages. These extensions test the robustness of the
finite-event mechanism while keeping the exact assumptions visible.

The model may also be viewed as a state-dependent interaction network, which
connects the analysis with consensus under prescribed switching
\cite{JadbabaieLinMorse2003,Moreau2005,OlfatiSaberFaxMurray2007} and
bounded-confidence dynamics
\cite{HegselmannKrause2002,BlondelHendrickxTsitsiklis2009}. Its second-order
transport structure, hard cut-off, and local normalization distinguish it
from those systems and from cluster prediction for smooth Cucker--Smale
interactions \cite{HaKimParkZhang2019}. The present results concern finite
particle systems; no thermodynamic or mean-field limit is asserted.

Section~\ref{sec:model} defines the model and states the main results.
Section~\ref{sec:order} proves exact terminal-cluster prediction in the
expansive regime. Section~\ref{sec:path} derives the path criterion and sharp
threshold. Section~\ref{sec:spectral} treats robustness and extensions.
Section~\ref{sec:numerics} gives numerical checks, and
Section~\ref{sec:discussion} discusses the finite-size interpretation and
open problems.

\section{Model and definitions}
\label{sec:model}

\subsection{Finite-range normalized local alignment}

Consider $N\ge2$ agents moving on the real line. Agent $i$ has position
$x_i(t)\in\R$ and velocity $v_i(t)\in\R$. Fix an interaction radius $r>0$
and an alignment rate $\kappa>0$. The active neighbor set of agent $i$ is
\begin{equation}
    \Nset_i(t)
    :=
    \left\{
        j\in\{1,\ldots,N\}\setminus\{i\}:
        |x_j(t)-x_i(t)|<r
    \right\},
    \label{eq:neighbors}
\end{equation}
and $n_i(t):=|\Nset_i(t)|$.

The dynamics is
\begin{equation}
    \dot x_i=v_i,
    \label{eq:x-dynamics}
\end{equation}
and, whenever $n_i(t)>0$,
\begin{equation}
    \dot v_i
    =
    \kappa
    \left(
        \frac{1}{n_i(t)}
        \sum_{j\in\Nset_i(t)}v_j
        -
        v_i
    \right).
    \label{eq:v-dynamics}
\end{equation}
If $n_i(t)=0$, we set
\begin{equation}
    \dot v_i=0.
    \label{eq:isolated}
\end{equation}

Equation~\eqref{eq:v-dynamics} is the self-excluding normalized local-average
rule used in the earlier model \cite{ChenZhangLim2013}. It is closely related
to the Motsch--Tadmor normalization \cite{MotschTadmor2011}, but the precise
normalization convention is fixed here by \eqref{eq:neighbors}--\eqref{eq:v-dynamics}
and will be used throughout the paper.

\begin{remark}[Strict cut-off convention]
The interaction condition is strict: a pair is active if and only if its
distance is strictly smaller than $r$. Thus a pair satisfying
$|x_i-x_j|=r$ is not an active edge. This convention is important in the
critical case where an active separation may converge to $r$ only as
$t\to\infty$.
\end{remark}

\subsection{Communication graph}

At time $t$, define the undirected communication graph
\[
    G(t)=(V,E(t)),
    \qquad
    V=\{1,\ldots,N\},
\]
with
\begin{equation}
    (i,j)\in E(t)
    \quad\Longleftrightarrow\quad
    |x_i(t)-x_j(t)|<r.
    \label{eq:graph}
\end{equation}

Because the graph depends on the evolving positions, the full system is a
state-dependent switching system. Between topology-change times the graph is
fixed and the velocity equation is linear. At a topology-change time the
positions and velocities are kept continuous and the active edge set is
updated according to \eqref{eq:graph}. If several pairs reach the interaction
boundary simultaneously, all corresponding edge changes are applied at the
same event time.

\begin{definition}[Edge loss]
An active edge $(i,j)$ is lost at a finite time $T$ if
\[
    |x_i(t)-x_j(t)|<r
    \quad\text{for }t<T\text{ sufficiently close to }T,
\]
and
\[
    |x_i(T)-x_j(T)|=r.
\]
Under the strict cut-off convention, $(i,j)\notin E(T)$.
\end{definition}

\begin{definition}[Fragmentation]
Suppose $G(0)$ is connected. The swarm is said to \emph{fragment in finite
time} if there exists a finite event time $T$ such that $G(T)$ is
disconnected.
\end{definition}

Edge loss and fragmentation are distinct notions: deleting a non-bridge edge
may change the communication topology without increasing the number of
connected components.

\begin{definition}[Terminal graph and asymptotic clusters]
A graph $G_\infty$ is called the \emph{terminal communication graph} if there
exists $T_\infty<\infty$ such that
\[
    G(t)=G_\infty
    \qquad\text{for all }t\ge T_\infty.
\]
The connected components of $G_\infty$ are called the \emph{terminal
clusters}. Their number is denoted by
\[
    K_\infty:=|\Comp(G_\infty)|.
\]
\end{definition}

For a nontrivial terminal component $C$, the fixed-graph analysis below will
show that all agents in $C$ converge to a common velocity. A singleton
component keeps the velocity it has at the time it becomes isolated.

\subsection{The expansive cone}

The principal deterministic theory of this paper concerns ordered positions
and velocities.

\begin{definition}[Expansive configuration]
A state $(x,v)$ belongs to the \emph{expansive cone} if
\begin{equation}
    x_1<x_2<\cdots<x_N
    \label{eq:ordered-x}
\end{equation}
and
\begin{equation}
    v_1\le v_2\le\cdots\le v_N.
    \label{eq:ordered-v}
\end{equation}
\end{definition}

For $i=1,\ldots,N-1$, define the adjacent spatial and velocity differences
\begin{equation}
    d_i:=x_{i+1}-x_i,
    \qquad
    q_i:=v_{i+1}-v_i.
    \label{eq:dq}
\end{equation}
Thus a state lies in the expansive cone precisely when $d_i>0$ and
$q_i\ge0$ for all $i$.

The terminology reflects the fact that, once velocity ordering is preserved,
\[
    \dot d_i=q_i\ge0.
\]
A central result of Section~\ref{sec:order} is that the cone
\eqref{eq:ordered-x}--\eqref{eq:ordered-v} is forward invariant. This implies
that all pairwise separations are nondecreasing and, consequently, the
communication graph can evolve only by irreversible edge deletion.

\subsection{Fixed-graph notation}

Let $G=(V,E)$ be an undirected communication graph. For a non-isolated
vertex $i$, let $\deg_G(i)$ denote its graph degree and define
\[
    (P_G)_{ij}
    =
    \begin{cases}
        1/\deg_G(i), & (i,j)\in E,\\
        0, & \text{otherwise}.
    \end{cases}
\]
For an isolated vertex we set $(P_G)_{ii}=1$ and all other entries in that
row equal to zero. Then $P_G$ is row stochastic:
\[
    P_G\one=\one.
\]
We define the random-walk Laplacian
\begin{equation}
    L_G:=I-P_G.
    \label{eq:LG}
\end{equation}
On every interval on which $G(t)\equiv G$,
\begin{equation}
    \dot v=-\kappa L_Gv.
    \label{eq:fixed-graph-v}
\end{equation}

For a matrix with semisimple zero eigenvalue, $L_G^\#$ will denote its group
inverse. Section~\ref{sec:fixed} will show that the group inverse is well defined here and
that, for an active edge $(i,j)$ with $i<j$,
\begin{equation}
    S_{ij}^{G}
    :=
    x_j-x_i
    +
    \frac{1}{\kappa}
    (e_j-e_i)^TL_G^\#v
    \label{eq:terminal-separation}
\end{equation}
is the terminal relative separation predicted by the current fixed graph.

\subsection{Main results and logical structure}

The paper has two principal finite-size conclusions. First, every initial
state in the expansive cone has an exactly predictable terminal clustering.
Velocity order is preserved, all separations are nondecreasing, and the graph
undergoes only finitely many edge deletions
(Theorem~\ref{thm:order-preservation}). At a current state $(x,v,G)$, the
graph is terminal precisely when
\begin{equation}
 S_{ij}^G\le r\qquad\text{for every }(i,j)\in E(G).
 \label{eq:main-terminal-test}
\end{equation}
If this fails, the smallest frozen hitting time is the next actual event
(Theorem~\ref{thm:terminal-graph-next-event}). Repeating this test terminates
after at most $|E(0)|$ events and returns the actual terminal graph and
cluster velocities (Theorem~\ref{thm:finite-cluster-prediction}).

Second, for an initial path with nondecreasing velocities, write
$d_i^0=x_{i+1}^0-x_i^0$, $q_i^0=v_{i+1}^0-v_i^0$, and $m=N-1$. Let $C_m$
be the tridiagonal matrix with off-diagonal entries $-1$, endpoint diagonal
entries $3$, and interior diagonal entries $2$; for $m=1$, set $C_1=[4]$.
Then
\begin{equation}
 d^*=d^0+\frac{2}{\kappa}C_m^{-1}q^0
 \label{eq:main-path-limit}
\end{equation}
is the frozen terminal-gap vector. Finite-time fragmentation occurs if and
only if $\max_i d_i^*>r$ (Theorem~\ref{thm:path-fragmentation}). Equivalently,
the sharp threshold is
\begin{equation}
 \kappa_c=\max_i\frac{2(C_m^{-1}q^0)_i}{r-d_i^0},
 \label{eq:main-kappa-threshold}
\end{equation}
with fragmentation for $0<\kappa<\kappa_c$ and no finite-time
fragmentation for $\kappa\ge\kappa_c$
(Theorem~\ref{thm:kappa-critical}). Equality is a genuinely marginal case:
at least one gap approaches $r$ only as $t\to\infty$.

The proof chain for the first conclusion is
\[
 \text{velocity order}
 \Longrightarrow \text{monotone separations}
 \Longrightarrow \text{irreversible deletion}
 \Longrightarrow \text{finite event recursion}.
\]
The fixed-graph group inverse enters only at the last implication, where it
computes the remaining relative displacement. Section~\ref{sec:order}
follows this chain without separating standard fixed-graph material from its
role in the switching argument. Section~\ref{sec:path} then specializes the
recursion to paths. Section~\ref{sec:spectral} records exactly how far the
hypotheses are extended and where the method ceases to apply.

The theory is discrete and does not establish a mean-field limit. Exact
cluster prediction is proved for expansive states and for a controlled class
that enters the expansive cone before leaving its initial path cell. For
arbitrary non-monotone data, crossings and edge creation destroy the monotone
event structure; that switching problem remains open.

\section{Exact prediction in the expansive regime}
\label{sec:order}

The proof follows the dependency chain stated in Section~\ref{sec:model}.
We first establish the invariant velocity order and the resulting monotone
graph evolution. We then compute the accumulated motion on one fixed graph
and use it to select the next actual event. Iteration gives the terminal
clusters and their limiting dynamics.

\subsection{Velocity convex-hull contraction}

We begin with a basic estimate that does not require ordered initial data.

\begin{proposition}[Velocity convex-hull contraction]
\label{prop:velocity-hull}
Let $(x(t),v(t))$ be a piecewise-classical solution of
\eqref{eq:x-dynamics}--\eqref{eq:isolated}.  Define
\[
    V_{\max}(t):=\max_{1\le i\le N}v_i(t),
    \qquad
    V_{\min}(t):=\min_{1\le i\le N}v_i(t).
\]
Then $V_{\max}$ is nonincreasing and $V_{\min}$ is nondecreasing.  In
particular,
\begin{equation}
    D_v(t):=V_{\max}(t)-V_{\min}(t)
    \le D_v(0)
    \qquad\text{for all }t\ge0.
    \label{eq:velocity-diameter}
\end{equation}
\end{proposition}

\begin{proof}
Fix a time interval on which the communication graph is constant.  If
$v_i(t)=V_{\max}(t)$ and $i$ is non-isolated, then every neighbor velocity is
at most $V_{\max}(t)$, hence
\[
    \dot v_i
    =
    \kappa
    \left(
        \frac{1}{n_i}\sum_{j\in\Nset_i}v_j-v_i
    \right)
    \le0.
\]
If $i$ is isolated, then $\dot v_i=0$.  This shows $\dot v_i(t)\le0$ at every
index $i$ attaining the maximum at time $t$, whether or not that maximizer is
unique.  Since $V_{\max}$ is the pointwise maximum of finitely many
differentiable functions, its upper right Dini derivative equals the largest
value of $\dot v_i(t)$ over all currently maximizing indices $i$; as every
such value is nonpositive, so is the Dini derivative of $V_{\max}$.  The
corresponding argument at a minimizing index, applied to every minimizer
simultaneously, gives a nonnegative lower right Dini derivative for
$V_{\min}$.  Since the velocities remain continuous at topology-change times,
the monotonicity extends across the switching events.
\end{proof}

\begin{remark}
Proposition~\ref{prop:velocity-hull} replaces the kinetic-energy
monotonicity that is unavailable for the present normalized dynamics.
Ordinary momentum is generally not conserved, but the velocity convex hull
is nevertheless forward invariant.
\end{remark}

\subsection{A one-dimensional sliding-neighborhood lemma}

The proof of order preservation relies on a geometric property specific to
one dimension.  Throughout this subsection assume
\[
    x_1<x_2<\cdots<x_N.
\]
For each $i$, introduce the radius-$r$ index window including the center
itself,
\begin{equation}
    I_i
    :=
    \{j:|x_j-x_i|<r\}.
    \label{eq:inclusive-window}
\end{equation}
Since the positions are ordered, there exist indices $\ell_i\le i\le R_i$
such that
\[
    I_i=\{\ell_i,\ell_i+1,\ldots,R_i\}.
\]

\begin{lemma}[Sliding of one-dimensional metric neighborhoods]
\label{lem:sliding-window}
For $i=1,\ldots,N-1$,
\begin{equation}
    \ell_{i+1}\ge\ell_i,
    \qquad
    R_{i+1}\ge R_i.
    \label{eq:window-monotonicity}
\end{equation}
\end{lemma}

\begin{proof}
Since $x_{i+1}>x_i$,
\[
    x_{i+1}-r>x_i-r,
    \qquad
    x_{i+1}+r>x_i+r.
\]
Moving the center from $x_i$ to $x_{i+1}$ therefore shifts both endpoints
of the open interaction interval to the right.  Because the positions are
strictly ordered, the smallest index contained in the interval cannot
decrease and the largest contained index cannot decrease.
\end{proof}

The next lemma is the key comparison at a boundary of the ordered-velocity
cone.

\begin{lemma}[Boundary acceleration inequality]
\label{lem:boundary-acceleration}
Assume
\[
    x_1<\cdots<x_N,
    \qquad
    v_1\le\cdots\le v_N,
\]
and suppose that for some $i\in\{1,\ldots,N-1\}$,
\[
    v_i=v_{i+1}.
\]
Then
\begin{equation}
    \dot v_{i+1}-\dot v_i\ge0.
    \label{eq:boundary-acceleration}
\end{equation}
The conclusion holds whether or not $i$ and $i+1$ are neighbors.
\end{lemma}

\begin{proof}
Write $v_i=v_{i+1}=c$.

First suppose that $x_{i+1}-x_i\ge r$.  Since $i$ and $i+1$ are adjacent in
the spatial ordering, every neighbor of $i$ has index smaller than $i$ and
therefore velocity at most $c$.  Hence $\dot v_i\le0$, with equality if $i$
is isolated.  Similarly, every neighbor of $i+1$ has index larger than
$i+1$ and velocity at least $c$, so $\dot v_{i+1}\ge0$.  Thus
\eqref{eq:boundary-acceleration} follows.

Now suppose that $x_{i+1}-x_i<r$.  Then $i$ and $i+1$ belong to both
inclusive windows $I_i$ and $I_{i+1}$.  The self-excluding neighbor-velocity
multisets may be written as
\[
    M_i
    =
    \{v_j:\ell_i\le j\le R_i,\ j\ne i\},
\]
and
\[
    M_{i+1}
    =
    \{v_j:\ell_{i+1}\le j\le R_{i+1},\ j\ne i+1\}.
\]
The multiset $M_i$ contains $v_{i+1}=c$ whereas $M_{i+1}$ contains
$v_i=c$.  We may therefore identify these two equal entries as a common
element.  By Lemma~\ref{lem:sliding-window}, passing from $M_i$ to
$M_{i+1}$ removes only entries on the far left and adds only entries on the
far right.  Because the velocities are nondecreasing with the index, each
removed entry is no larger than every entry that remains, whereas each added
entry is no smaller than every entry already present.

Removing a minimum element from a nonempty finite multiset cannot decrease
its arithmetic mean, and adding an element no smaller than the current
maximum cannot decrease the mean.  Applying these operations successively
gives
\[
    \operatorname{mean}(M_{i+1})
    \ge
    \operatorname{mean}(M_i).
\]
Since $v_i=v_{i+1}=c$,
\[
    \dot v_{i+1}-\dot v_i
    =
    \kappa
    \left[
        \operatorname{mean}(M_{i+1})
        -
        \operatorname{mean}(M_i)
    \right]
    \ge0.
\]
\end{proof}

\subsection{Positive dynamics for adjacent velocity differences}

The boundary inequality above has an equivalent positive-systems
interpretation that gives a convenient rigorous proof of invariance.

Fix an interval of time on which the communication graph is a constant graph
$G$ generated by an ordered spatial configuration.  Let $H_G$ be the matrix
such that
\begin{equation}
    \dot v=\kappa H_Gv.
    \label{eq:HG}
\end{equation}
For a non-isolated vertex,
\[
    (H_Gv)_i
    =
    \frac{1}{\deg_G(i)}
    \sum_{j\sim i}v_j-v_i,
\]
while an isolated vertex corresponds to a zero row.  In either case,
\begin{equation}
    H_G\one=0.
    \label{eq:HG-one}
\end{equation}

Let $Q\in\R^{(N-1)\times N}$ be the adjacent-difference matrix
\[
    Qv
    =
    (v_2-v_1,\ldots,v_N-v_{N-1})^T.
\]
Then $\ker Q=\operatorname{span}\{\one\}$.  Equation
\eqref{eq:HG-one} implies that $QH_Gv$ depends only on $Qv$.  Consequently
there is a unique linear map $M_G$ on $\R^{N-1}$ such that
\begin{equation}
    QH_G=M_GQ.
    \label{eq:intertwining}
\end{equation}
Thus, with
\[
    q:=Qv,
\]
the adjacent velocity differences satisfy
\begin{equation}
    \dot q=\kappa M_Gq.
    \label{eq:q-general-fixed}
\end{equation}

\begin{lemma}[Metzler structure]
\label{lem:metzler}
For every fixed communication graph $G$ generated by an ordered
one-dimensional configuration, the matrix $M_G$ in
\eqref{eq:intertwining} is Metzler; that is,
\[
    (M_G)_{ik}\ge0
    \qquad\text{whenever }i\ne k.
\]
\end{lemma}

\begin{proof}
Fix $k\in\{1,\ldots,N-1\}$ and choose a nondecreasing velocity vector
\[
    v_j=
    \begin{cases}
        0,&j\le k,\\
        1,&j\ge k+1.
    \end{cases}
\]
Then $Qv=e_k$.  For every $i\ne k$ we have
$v_i=v_{i+1}$, and therefore Lemma~\ref{lem:boundary-acceleration} gives
\[
    (QH_Gv)_i\ge0.
\]
Using \eqref{eq:intertwining},
\[
    (QH_Gv)_i
    =
    (M_Ge_k)_i
    =
    (M_G)_{ik}.
\]
Hence every off-diagonal entry of $M_G$ is nonnegative.
\end{proof}

\begin{corollary}[Positivity on a fixed graph]
\label{cor:fixed-positive}
On a time interval on which the graph is fixed,
\[
    q(t_0)\ge0
    \quad\Longrightarrow\quad
    q(t)\ge0
    \qquad\text{for all }t\ge t_0
\]
as long as that graph remains active.
\end{corollary}

\begin{proof}
A Metzler matrix generates a positive semigroup; see, e.g.,
\cite{FarinaRinaldi2000,BermanPlemmons1994}.  Indeed, choose
$\alpha>0$ sufficiently large that $M_G+\alpha I$ is entrywise
nonnegative.  Then
\[
    e^{\kappa M_Gs}
    =
    e^{-\kappa\alpha s}
    e^{\kappa(M_G+\alpha I)s}
\]
is entrywise nonnegative for every $s\ge0$, since the power series for the
second exponential has only nonnegative terms.  Equation
\eqref{eq:q-general-fixed} now gives the claim.
\end{proof}

\subsection{Forward invariance of the expansive cone}

We now pass from a fixed communication graph to the full state-dependent
system.

\begin{theorem}[Order preservation and irreversible graph evolution]
\label{thm:order-preservation}
Suppose the initial data satisfy
\begin{equation}
    x_1(0)<x_2(0)<\cdots<x_N(0),
    \label{eq:thm-ordered-x0}
\end{equation}
and
\begin{equation}
    v_1(0)\le v_2(0)\le\cdots\le v_N(0).
    \label{eq:thm-ordered-v0}
\end{equation}
Then the event-driven solution of
\eqref{eq:x-dynamics}--\eqref{eq:isolated} exists for all $t\ge0$ and
satisfies
\begin{equation}
    v_1(t)\le v_2(t)\le\cdots\le v_N(t)
    \qquad\text{for all }t\ge0.
    \label{eq:velocity-order-all-time}
\end{equation}
Moreover, for every $i<j$,
\begin{equation}
    x_j(t)-x_i(t)
    \quad\text{is nondecreasing in }t.
    \label{eq:pair-separation-monotone}
\end{equation}
Consequently,
\begin{equation}
    E(t_2)\subseteq E(t_1)
    \qquad
    \text{whenever }t_2\ge t_1,
    \label{eq:edge-monotonicity}
\end{equation}
and every topology change consists only of the deletion of one or more
active edges.
\end{theorem}

\begin{proof}
Let $q_i=v_{i+1}-v_i$.  On the initial fixed-graph interval,
Corollary~\ref{cor:fixed-positive} and
\eqref{eq:thm-ordered-v0} imply
\[
    q_i(t)\ge0
\]
up to the first topology-change time.  Therefore
\[
    \frac{d}{dt}(x_{i+1}-x_i)=q_i(t)\ge0,
\]
so every adjacent gap remains at least its strictly positive initial value.
In particular, particles cannot cross and the spatial ordering
\eqref{eq:thm-ordered-x0} is preserved.

For arbitrary $i<j$,
\[
    v_j-v_i
    =
    \sum_{k=i}^{j-1}q_k
    \ge0,
\]
hence
\[
    \frac{d}{dt}(x_j-x_i)=v_j-v_i\ge0.
\]
Thus every pairwise separation is nondecreasing before the first topology
change.  A pair that is not an edge cannot therefore enter the interaction
radius.  Hence the first topology change, if it occurs, can only delete
currently active edges.

At such an event time the positions and velocities are continuous.  In
particular $q_i\ge0$ still holds immediately after the event.  The new graph
is again generated by the same ordered one-dimensional configuration, so
Corollary~\ref{cor:fixed-positive} applies on the next fixed-graph interval.
Iterating proves \eqref{eq:velocity-order-all-time},
\eqref{eq:pair-separation-monotone}, and \eqref{eq:edge-monotonicity}
through every switching event.

Finally, Proposition~\ref{prop:velocity-hull} bounds all velocities uniformly
by the initial velocity convex hull, so positions cannot blow up in finite
time.  Since, as shown below, only finitely many topology changes can occur,
the piecewise-classical construction extends for all $t\ge0$.
\end{proof}

\begin{corollary}[No collision, no edge recovery, and finite switching]
\label{cor:finite-switching}
Under the hypotheses of Theorem~\ref{thm:order-preservation}:

\begin{enumerate}[label=(\roman*)]
\item the spatial ordering remains strict:
\[
    x_1(t)<x_2(t)<\cdots<x_N(t)
    \qquad\text{for all }t\ge0;
\]

\item once an edge is lost, it can never be recovered;

\item no new communication edge can be created;

\item the number of distinct topology-change times is at most
\[
    |E(0)|
    \le\frac{N(N-1)}2.
\]
In particular, the solution has no Zeno accumulation of switching times.
\end{enumerate}
\end{corollary}

\begin{proof}
Part (i) follows because every adjacent gap is nondecreasing and initially
strictly positive.  Parts (ii) and (iii) are immediate from the
nondecreasing pairwise separations and the strict cut-off rule.  At every
genuine topology-change time at least one previously active edge is deleted,
and by part (ii) that edge can never return.  Since the initial graph has
only finitely many edges, there can be at most $|E(0)|$ such event times.
\end{proof}

\begin{remark}[Why one dimension matters]
Theorem~\ref{thm:order-preservation} uses the total spatial ordering of the
line twice: metric neighborhoods slide monotonically in index space, and
velocity ordering converts directly into monotonicity of every pairwise
separation.  Neither mechanism has a direct analogue for a generic
configuration in two or more spatial dimensions.
\end{remark}

\subsection{Accumulated motion on a fixed communication graph}
\label{sec:fixed}

Theorem~\ref{thm:order-preservation} shows that, in the expansive regime,
the communication graph changes only by edge deletion and does so only
finitely many times.  We now analyze one interval on which the graph is
fixed.  The fixed-graph dynamics is linear, and its long-time relative
motion can be described exactly by the group inverse of the random-walk
Laplacian.

Throughout this section, let $G=(V,E)$ be a fixed undirected communication
graph and write
\[
    L_G=I-P_G
\]
as in \eqref{eq:LG}.  The graph is allowed to be disconnected and to contain
isolated vertices.

\subsubsection{Spectral structure of the normalized graph Laplacian}

Let $C\subseteq V$ be a connected component with at least two vertices.
Write $A_C$ for its adjacency matrix and
\[
    D_C:=\diag(\deg_C(i):i\in C)
\]
for its degree matrix.  On this component,
\[
    P_C=D_C^{-1}A_C,
    \qquad
    L_C=I-D_C^{-1}A_C.
\]

\begin{lemma}[Similarity to a symmetric normalized Laplacian]
\label{lem:laplacian-similarity}
For every nontrivial connected component $C$,
\begin{equation}
    D_C^{1/2}L_CD_C^{-1/2}
    =
    I-D_C^{-1/2}A_CD_C^{-1/2}.
    \label{eq:symmetric-similarity}
\end{equation}
Consequently, $L_C$ is diagonalizable with real nonnegative spectrum,
and $0$ is a simple eigenvalue.
\end{lemma}

\begin{proof}
Identity \eqref{eq:symmetric-similarity} follows by direct multiplication.
The matrix on the right-hand side is the symmetric normalized Laplacian of
the undirected connected graph $C$ \cite{Chung1997}.  It is symmetric
positive semidefinite, and its nullspace is one-dimensional.  Similarity
preserves eigenvalues and diagonalizability.
\end{proof}

For an isolated vertex $i$, our convention $P_{ii}=1$ gives the
one-dimensional block $L_{\{i\}}=0$.  Hence the zero eigenvalue of the full
matrix $L_G$ is semisimple, with multiplicity equal to the number of
connected components of $G$.

For every nontrivial connected component $C$, define
\begin{equation}
    \pi_i^C
    :=
    \frac{\deg_C(i)}
    {\sum_{j\in C}\deg_C(j)},
    \qquad i\in C.
    \label{eq:stationary-pi}
\end{equation}
Then $(\pi^C)^TP_C=(\pi^C)^T$ and $(\pi^C)^T\one_C=1$.  For a singleton
component $C=\{i\}$, set $\pi^C=1$.  Define the component projection
\begin{equation}
    \Pi_C:=\one_C(\pi^C)^T,
    \label{eq:component-projection}
\end{equation}
and let $\Pi_G$ be the block-diagonal matrix whose blocks are the
$\Pi_C$ over the connected components of $G$.

\begin{proposition}[Fixed-graph consensus projection]
\label{prop:fixed-consensus}
For a fixed graph $G$,
\begin{equation}
    e^{-\kappa L_Gt}\longrightarrow\Pi_G
    \qquad\text{as }t\to\infty.
    \label{eq:semigroup-to-projection}
\end{equation}
If $C$ is a nontrivial connected component, then all velocities in $C$
converge to the degree-weighted value
\begin{equation}
    V_C^\infty
    =
    (\pi^C)^Tv^0_C
    =
    \frac{\sum_{i\in C}\deg_C(i)v_i^0}
    {\sum_{i\in C}\deg_C(i)}.
    \label{eq:degree-weighted-consensus}
\end{equation}
For a singleton component, the velocity remains constant.
\end{proposition}

\begin{proof}
By Lemma~\ref{lem:laplacian-similarity}, every nonzero eigenvalue of a
nontrivial connected block $L_C$ is strictly positive and the zero
eigenspace is spanned by $\one_C$.  The corresponding left nullvector is
$\pi^C$.  Therefore the semigroup converges to
$\one_C(\pi^C)^T=\Pi_C$.  The singleton statement follows from
$L_{\{i\}}=0$.
\end{proof}

\begin{corollary}[Degree-weighted momentum on a fixed component]
\label{cor:weighted-momentum}
If $C$ is a nontrivial connected component and the graph remains fixed, then
\begin{equation}
    \sum_{i\in C}\deg_C(i)v_i(t)
    \label{eq:weighted-momentum}
\end{equation}
is constant in time.
\end{corollary}

\begin{proof}
Since $(\pi^C)^TL_C=0$,
\[
    \frac{d}{dt}(\pi^C)^Tv_C
    =
    -\kappa(\pi^C)^TL_Cv_C
    =0.
\]
Multiplying by the constant component volume
$\sum_{i\in C}\deg_C(i)$ gives \eqref{eq:weighted-momentum}.
\end{proof}

\begin{remark}
The conserved quantity in Corollary~\ref{cor:weighted-momentum} depends on
the current graph through the degrees.  It is therefore not conserved across
a topology-change event.  This is one manifestation of the nonsymmetry of
the normalized interaction rule.
\end{remark}

\subsubsection{The group inverse and exact fixed-graph trajectories}

Because the zero eigenvalue of $L_G$ is semisimple, the group inverse
$L_G^\#$ is well defined \cite{CampbellMeyer2009}.  Since
$-L_G=P_G-I$ is the generator of a continuous-time Markov chain on the
vertex set, $L_G^\#$
is precisely the group inverse that governs the fundamental quantities of
finite Markov chains in the sense of Meyer \cite{Meyer1975}; here it plays
the analogous role for accumulated relative displacements.  It is
characterized by
\begin{equation}
    L_GL_G^\#
    =
    L_G^\#L_G
    =
    I-\Pi_G,
    \qquad
    L_G^\#\Pi_G
    =
    \Pi_GL_G^\#
    =
    0.
    \label{eq:group-inverse-identities}
\end{equation}

\begin{lemma}[Integral representation of the group inverse]
\label{lem:group-inverse-integral}
For every fixed communication graph $G$,
\begin{equation}
    L_G^\#
    =
    \int_0^\infty
    \left(e^{-L_Gs}-\Pi_G\right)\,ds,
    \label{eq:group-inverse-integral}
\end{equation}
and, for every $t\ge0$,
\begin{equation}
    \int_0^t
    \left(e^{-\kappa L_Gs}-\Pi_G\right)\,ds
    =
    \frac{1}{\kappa}
    L_G^\#
    \left(I-e^{-\kappa L_Gt}\right).
    \label{eq:finite-group-inverse-integral}
\end{equation}
\end{lemma}

\begin{proof}
Diagonalize $L_G$ componentwise.  The operator $e^{-L_Gs}-\Pi_G$ vanishes
on the zero eigenspace and acts by $e^{-\lambda s}$ on every positive
eigenmode.  Integrating gives multiplication by $1/\lambda$ on each
positive eigenspace and zero on the nullspace, which is precisely the group
inverse.  This proves \eqref{eq:group-inverse-integral}.

For \eqref{eq:finite-group-inverse-integral}, define
\[
    F(t)
    :=
    \frac{1}{\kappa}
    L_G^\#
    \left(I-e^{-\kappa L_Gt}\right).
\]
Using \eqref{eq:group-inverse-identities},
\[
    F'(t)
    =
    L_G^\#L_Ge^{-\kappa L_Gt}
    =
    (I-\Pi_G)e^{-\kappa L_Gt}
    =
    e^{-\kappa L_Gt}-\Pi_G,
\]
and $F(0)=0$.
\end{proof}

\begin{proposition}[Exact fixed-graph trajectory]
\label{prop:fixed-trajectory}
Suppose the graph is held fixed at $G$ and the state at time $0$ is
$(x^0,v^0)$.  Then
\begin{equation}
    v(t)=e^{-\kappa L_Gt}v^0
    \label{eq:exact-fixed-v}
\end{equation}
and
\begin{equation}
    x(t)
    =
    x^0
    +
    t\,\Pi_Gv^0
    +
    \frac{1}{\kappa}
    L_G^\#
    \left(I-e^{-\kappa L_Gt}\right)v^0.
    \label{eq:exact-fixed-x}
\end{equation}
\end{proposition}

\begin{proof}
Equation \eqref{eq:exact-fixed-v} is the solution of
\eqref{eq:fixed-graph-v}.  Integrating $\dot x=v$ and decomposing
\[
    e^{-\kappa L_Gs}
    =
    \Pi_G+
    \left(e^{-\kappa L_Gs}-\Pi_G\right)
\]
gives \eqref{eq:exact-fixed-x} by
Lemma~\ref{lem:group-inverse-integral}.
\end{proof}

The term $t\,\Pi_Gv^0$ in \eqref{eq:exact-fixed-x} is the ballistic motion
of the component consensus velocities.  It cancels from relative positions
inside a connected component.

\subsubsection{Frozen terminal separations}

Let $(i,j)\in E(G)$ with $i<j$, and define
\[
    b_{ij}:=e_j-e_i.
\]
Since an active edge joins vertices in the same connected component, the
corresponding rows of $\Pi_G$ are equal and therefore
\begin{equation}
    b_{ij}^T\Pi_G=0.
    \label{eq:bPi-zero}
\end{equation}
Define the signed separation
\[
    \Delta_{ij}(t):=x_j(t)-x_i(t).
\]

\begin{proposition}[Exact relative-position formula]
\label{prop:relative-position}
For the fixed-graph continuation generated by $G$,
\begin{equation}
    \Delta_{ij}(t)
    =
    \Delta_{ij}^0
    +
    \frac{1}{\kappa}
    b_{ij}^TL_G^\#
    \left(I-e^{-\kappa L_Gt}\right)v^0.
    \label{eq:relative-position}
\end{equation}
In particular,
\begin{equation}
    S_{ij}^G
    :=
    \lim_{t\to\infty}\Delta_{ij}(t)
    =
    \Delta_{ij}^0
    +
    \frac{1}{\kappa}
    b_{ij}^TL_G^\#v^0.
    \label{eq:frozen-terminal-separation}
\end{equation}
\end{proposition}

\begin{proof}
Apply $b_{ij}^T$ to \eqref{eq:exact-fixed-x}.  The ballistic term vanishes
by \eqref{eq:bPi-zero}, which gives
\eqref{eq:relative-position}.  Letting $t\to\infty$ and using
$L_G^\#\Pi_G=0$ yields \eqref{eq:frozen-terminal-separation}.
\end{proof}

We call $S_{ij}^G$ the \emph{frozen-graph terminal separation}.  It is the
limiting signed distance predicted if the current communication graph were
held fixed indefinitely.  In the actual state-dependent system, another
edge may be deleted before this limiting state is reached; therefore
$S_{ij}^G>r$ is not, by itself, a statement that the same edge must
eventually disappear after arbitrary intervening topology changes.  Its
exact role is to determine stability of the current graph and the next
topology event.

\subsubsection{Exact frozen edge-loss criterion in the expansive regime}

We now assume that the current state lies in the expansive cone.  By
Theorem~\ref{thm:order-preservation}, the adjacent velocity differences are
nonnegative.  The same positive-system argument applies to the frozen
continuation of the current graph, so for every $i<j$,
\begin{equation}
    \frac{d}{dt}\Delta_{ij}(t)
    =
    v_j(t)-v_i(t)
    \ge0.
    \label{eq:frozen-separation-monotone}
\end{equation}

\begin{theorem}[Frozen edge-loss criterion]
\label{thm:frozen-edge-loss}
Let $(i,j)\in E(G)$ with $i<j$, and suppose the current state is in the
expansive cone.  For the fixed-graph continuation generated by $G$:

\begin{enumerate}[label=(\roman*)]
\item if $S_{ij}^G<r$, then
\[
    \Delta_{ij}(t)<r
    \qquad\text{for every finite }t;
\]

\item if $S_{ij}^G=r$, then
\[
    \Delta_{ij}(t)<r
    \qquad\text{for every finite }t,
    \qquad
    \Delta_{ij}(t)\uparrow r
    \quad\text{as }t\to\infty;
\]

\item if $S_{ij}^G>r$, then there exists a unique finite time
$T_{ij}^G>0$ satisfying
\begin{equation}
    \Delta_{ij}^0
    +
    \frac{1}{\kappa}
    b_{ij}^TL_G^\#
    \left(I-e^{-\kappa L_GT_{ij}^G}\right)v^0
    =
    r.
    \label{eq:edge-hitting-equation}
\end{equation}
\end{enumerate}

Hence an active edge reaches the strict interaction boundary in finite
\emph{frozen} time if and only if
\begin{equation}
    S_{ij}^G>r.
    \label{eq:frozen-edge-iff}
\end{equation}
\end{theorem}

\begin{proof}
The initial edge is active, so $\Delta_{ij}^0<r$.  By
\eqref{eq:frozen-separation-monotone}, $\Delta_{ij}$ is nondecreasing.

If $S_{ij}^G<r$, monotonicity and convergence to $S_{ij}^G$ give part (i).

Suppose $S_{ij}^G=r$.  If $\Delta_{ij}(T)=r$ for some finite $T$, then
monotonicity together with the limiting value $r$ would force
$\Delta_{ij}(t)=r$ for every $t\ge T$.  The function
$\Delta_{ij}(t)$ is real analytic under the fixed-graph linear system, so
being constant on a nontrivial interval would imply that it is constant for
all $t$.  This contradicts $\Delta_{ij}^0<r$.  Thus the boundary is
approached only asymptotically, proving part (ii).

If $S_{ij}^G>r$, continuity gives at least one finite time at which
$\Delta_{ij}=r$.  Monotonicity shows that two distinct isolated crossings
are impossible; an interval of equality is ruled out by the same analyticity
argument.  Hence the hitting time is unique and satisfies
\eqref{eq:edge-hitting-equation}.
\end{proof}

\subsubsection{Stability of the current graph and the next event}

The frozen criterion becomes an exact statement about the actual switching
system when applied to the earliest candidate edge loss.

\begin{theorem}[Terminal-graph criterion and next topology event]
\label{thm:terminal-graph-next-event}
Assume the current state lies in the expansive cone and let $G$ be the
current communication graph.

\begin{enumerate}[label=(\roman*)]
\item The graph $G$ is terminal, meaning that the actual communication graph
remains equal to $G$ for all future time, if and only if
\begin{equation}
    S_{ij}^G\le r
    \qquad
    \text{for every }(i,j)\in E(G).
    \label{eq:terminal-graph-criterion}
\end{equation}

\item If at least one active edge satisfies $S_{ij}^G>r$, define
\[
    \E_+(G)
    :=
    \{(i,j)\in E(G):S_{ij}^G>r\}.
\]
For each edge in $\E_+(G)$ let $T_{ij}^G$ be the unique frozen hitting time
from Theorem~\ref{thm:frozen-edge-loss}, and define
\begin{equation}
    T_*(G)
    :=
    \min_{(i,j)\in\E_+(G)}T_{ij}^G.
    \label{eq:next-event-time}
\end{equation}
Then the actual graph remains equal to $G$ on $[0,T_*(G))$, and the next
actual topology event occurs at $T_*(G)$.  Exactly those active edges whose
frozen hitting time equals $T_*(G)$ are deleted at that event.
\end{enumerate}
\end{theorem}

\begin{proof}
If \eqref{eq:terminal-graph-criterion} holds, then
Theorem~\ref{thm:frozen-edge-loss} shows that no active edge reaches the
interaction boundary in finite frozen time.  Theorem~\ref{thm:order-preservation}
shows that no non-edge can enter the interaction range.  Hence no topology
event can occur, the frozen continuation is the actual trajectory, and $G$
is terminal.

Conversely, if some $S_{ij}^G>r$, at least one frozen candidate hitting time
is finite.  Let $T_*$ be the minimum.  No new edge can appear before $T_*$
by Theorem~\ref{thm:order-preservation}.  No active edge with
$S_{ij}^G\le r$ can disappear before $T_*$ by
Theorem~\ref{thm:frozen-edge-loss}, and no edge with $S_{ij}^G>r$ can
disappear before its own frozen hitting time.  Therefore no topology change
occurs on $[0,T_*)$, so the actual trajectory on that interval is exactly the
fixed-graph continuation.  At $T_*$, all and only the edges attaining the
minimum reach the interaction boundary and are deleted simultaneously.
\end{proof}

\begin{remark}[Edge loss versus fragmentation]
Theorem~\ref{thm:terminal-graph-next-event} predicts topology changes, not
only connectivity changes.  An edge deletion need not fragment the swarm:
deleting a non-bridge can leave the graph connected.  Finite-time
fragmentation occurs precisely when an event increases the number of
connected components.  For a path graph every edge is a bridge, so edge loss
and fragmentation coincide; this special case is developed in
Section~\ref{sec:path}.
\end{remark}

\begin{remark}[Why the recursion is necessary]
An edge with $S_{ij}^G>r$ is guaranteed to hit the boundary under the frozen
continuation of $G$, but another edge may hit first and change the subsequent
dynamics.  Thus $S_{ij}^G>r$ should not be interpreted as an unconditional
statement that the same edge must eventually be lost in the full switching
system.  The exact prediction is obtained by taking the earliest frozen
hitting event, updating the graph, and recomputing the quantities for the new
topology.  Section~\ref{sec:prediction} formalizes this finite recursion.
\end{remark}

\subsection{Terminal-cluster recursion}
\label{sec:prediction}

The previous section determines, from the current state and communication
graph, whether another topology change must occur and, if so, exactly when
the next event occurs.  We now iterate that construction.  Because
Theorem~\ref{thm:order-preservation} makes every edge deletion irreversible,
the recursion terminates after finitely many events and gives the terminal
communication graph, the terminal cluster partition, and the asymptotic
velocity of every cluster.

\subsubsection{Event-driven recursion}

Assume throughout this section that the initial state lies in the expansive
cone.  Set
\[
    t_0:=0,
    \qquad
    x^{(0)}:=x(0),
    \qquad
    v^{(0)}:=v(0),
    \qquad
    G_0:=G(0).
\]
Suppose recursively that the state immediately after the $k$th topology
event is
\[
    \bigl(x^{(k)},v^{(k)},G_k\bigr)
\]
at absolute time $t_k$.  Let
\[
    L_k:=L_{G_k},
    \qquad
    \Pi_k:=\Pi_{G_k}.
\]
For every active edge $e=(i,j)\in E(G_k)$ with $i<j$, define
\begin{equation}
    S_e^{(k)}
    :=
    x_j^{(k)}-x_i^{(k)}
    +
    \frac{1}{\kappa}
    (e_j-e_i)^TL_k^\#v^{(k)}.
    \label{eq:recursive-terminal-separation}
\end{equation}

If
\begin{equation}
    S_e^{(k)}\le r
    \qquad\text{for every }e\in E(G_k),
    \label{eq:recursive-stop}
\end{equation}
then $G_k$ is terminal by
Theorem~\ref{thm:terminal-graph-next-event}, and the recursion stops.

Otherwise define the set of unstable frozen edges
\begin{equation}
    \E_k^+
    :=
    \{e\in E(G_k):S_e^{(k)}>r\}.
    \label{eq:unstable-edge-set}
\end{equation}
For $e=(i,j)\in\E_k^+$, let $\tau_e^{(k)}>0$ denote the unique solution of
\begin{equation}
    x_j^{(k)}-x_i^{(k)}
    +
    \frac{1}{\kappa}
    (e_j-e_i)^TL_k^\#
    \left(I-e^{-\kappa L_k\tau_e^{(k)}}\right)v^{(k)}
    =
    r.
    \label{eq:recursive-hitting-time}
\end{equation}
The next inter-event time is
\begin{equation}
    \tau_{k+1}
    :=
    \min_{e\in\E_k^+}\tau_e^{(k)},
    \qquad
    t_{k+1}:=t_k+\tau_{k+1}.
    \label{eq:recursive-next-time}
\end{equation}

The exact state at the next event is
\begin{equation}
    v^{(k+1)}
    =
    e^{-\kappa L_k\tau_{k+1}}v^{(k)}
    \label{eq:recursive-v-update}
\end{equation}
and
\begin{equation}
    x^{(k+1)}
    =
    x^{(k)}
    +
    \tau_{k+1}\Pi_kv^{(k)}
    +
    \frac{1}{\kappa}
    L_k^\#
    \left(I-e^{-\kappa L_k\tau_{k+1}}\right)v^{(k)}.
    \label{eq:recursive-x-update}
\end{equation}
Let
\begin{equation}
    D_{k+1}
    :=
    \left\{
        e\in\E_k^+:
        \tau_e^{(k)}=\tau_{k+1}
    \right\}
    \label{eq:simultaneous-deletion-set}
\end{equation}
be the set of edges that hit the interaction boundary simultaneously.
Under the strict cut-off convention,
\begin{equation}
    E(G_{k+1})
    =
    E(G_k)\setminus D_{k+1}.
    \label{eq:recursive-graph-update}
\end{equation}

\begin{remark}[Numerical implementation]
Although \eqref{eq:recursive-hitting-time} need not have an elementary
closed-form solution for a general graph, each candidate separation is
nondecreasing in the expansive regime and its finite hitting time is unique.
Thus each $\tau_e^{(k)}$ is the unique root of a one-dimensional monotone
equation.  The recursion is therefore exact at the level of the dynamical
system while remaining straightforward to implement numerically.
\end{remark}

\subsubsection{Exact terminal-cluster prediction}

\begin{theorem}[Finite-event topology and cluster prediction]
\label{thm:finite-cluster-prediction}
Suppose
\[
    x_1(0)<\cdots<x_N(0),
    \qquad
    v_1(0)\le\cdots\le v_N(0).
\]
Then the event-driven recursion
\eqref{eq:recursive-terminal-separation}--\eqref{eq:recursive-graph-update}
has the following properties.

\begin{enumerate}[label=(\roman*)]
\item At every stage $k$, the recursively constructed trajectory on
$[t_k,t_{k+1})$ coincides with the actual solution of
\eqref{eq:x-dynamics}--\eqref{eq:isolated}.

\item Whenever the stopping condition \eqref{eq:recursive-stop} fails,
$D_{k+1}$ is nonempty and
\begin{equation}
    E(G_{k+1})\subsetneq E(G_k).
    \label{eq:strict-edge-decrease}
\end{equation}

\item The recursion terminates after a finite number $M$ of topology-event
times satisfying
\begin{equation}
    M
    \le
    |E(G_0)|-|E(G_M)|
    \le
    |E(G_0)|.
    \label{eq:event-count-bound}
\end{equation}

\item The terminal graph produced by the recursion is exactly the terminal
communication graph of the full switching dynamics:
\begin{equation}
    G_M=G_\infty.
    \label{eq:predicted-terminal-graph}
\end{equation}
Consequently, the terminal cluster partition and cluster number are
\begin{equation}
    \Comp(G_\infty)=\Comp(G_M),
    \qquad
    K_\infty=|\Comp(G_M)|.
    \label{eq:predicted-clusters}
\end{equation}
\end{enumerate}
\end{theorem}

\begin{proof}
The proof is by induction over topology events.

At $k=0$, Theorem~\ref{thm:terminal-graph-next-event} states that if
\eqref{eq:recursive-stop} holds, then $G_0$ is already terminal.  Otherwise,
the actual graph remains equal to $G_0$ until the smallest frozen hitting
time $\tau_1$, and the state on that interval is exactly the fixed-graph
solution.  Equations \eqref{eq:recursive-v-update} and
\eqref{eq:recursive-x-update} therefore give the actual state at $t_1$.
Exactly the edges in $D_1$ hit the cut-off at that time and are deleted.

Theorem~\ref{thm:order-preservation} implies that the state at $t_1$ remains
in the expansive cone.  Hence the same argument applies with $G_1$ in place
of $G_0$.  Repeating establishes part (i) at every stage and shows that, when
the stopping condition fails, at least one edge is deleted.  This proves
\eqref{eq:strict-edge-decrease}.

No deleted edge can ever return by
Corollary~\ref{cor:finite-switching}.  Thus distinct nonterminal recursion
steps delete disjoint nonempty sets of edges.  The number of event times is
therefore at most the total number of deleted edges, giving
\eqref{eq:event-count-bound}.

Since the recursion must terminate, let $M$ be its final index.  At that
stage all active edges satisfy $S_e^{(M)}\le r$, so
Theorem~\ref{thm:terminal-graph-next-event} implies that $G_M$ remains
unchanged for all future time.  Because every preceding recursive segment
coincides with the actual trajectory, $G_M$ is exactly the terminal graph of
the full switching system.  Part (iv) follows.
\end{proof}

\begin{corollary}[Exact decision of fragmentation]
\label{cor:fragmentation-decision}
Suppose the hypotheses of Theorem~\ref{thm:finite-cluster-prediction} hold
and $G_0$ is connected.  Then finite-time fragmentation occurs if and only
if
\begin{equation}
    K_\infty>1.
    \label{eq:fragmentation-iff-K}
\end{equation}
Equivalently, the recursion decides fragmentation exactly from the initial
state.
\end{corollary}

\begin{proof}
If the terminal graph has more than one connected component, then, because
the graph begins connected and changes only at finitely many edge-deletion
events, there is a first event at which connectivity is lost.  Conversely,
once the graph becomes disconnected, irreversibility of edge deletion
prevents its components from reconnecting.  Hence finite-time fragmentation
is equivalent to $K_\infty>1$.
\end{proof}

\subsubsection{Asymptotic velocities and relative geometry}

Let
\[
    G_\infty=C_1\cup\cdots\cup C_{K_\infty}
\]
denote the terminal connected-component decomposition, and let $t_M$ be the
last topology-event time.  For a nontrivial component $C$, write
$\deg_C(i)$ for the degree of $i$ in the terminal graph.

\begin{corollary}[Terminal cluster velocities]
\label{cor:terminal-cluster-velocities}
For every nontrivial terminal component $C$,
\begin{equation}
    v_i(t)\longrightarrow V_C^\infty
    \qquad
    (i\in C),
    \label{eq:terminal-component-consensus}
\end{equation}
where
\begin{equation}
    V_C^\infty
    =
    \frac{
        \sum_{i\in C}\deg_C(i)v_i(t_M)
    }{
        \sum_{i\in C}\deg_C(i)
    }.
    \label{eq:terminal-cluster-velocity}
\end{equation}
If $C=\{i\}$ is a singleton, then
\begin{equation}
    V_C^\infty=v_i(t_M).
    \label{eq:singleton-terminal-velocity}
\end{equation}
Thus the recursion determines both the terminal partition and the
asymptotic bulk velocity of every cluster.
\end{corollary}

\begin{proof}
After $t_M$ the graph is fixed at $G_\infty$.  Apply
Proposition~\ref{prop:fixed-consensus} independently on each terminal
component.
\end{proof}

The same fixed-graph formula also determines the terminal geometry modulo
the common translational motion of each component.

\begin{corollary}[Asymptotic internal geometry]
\label{cor:terminal-geometry}
Let $s=t-t_M\ge0$.  Then
\begin{equation}
    x(t_M+s)
    -
    s\,\Pi_{G_\infty}v(t_M)
    \longrightarrow
    x(t_M)
    +
    \frac{1}{\kappa}
    L_{G_\infty}^\#v(t_M)
    \qquad
    \text{as }s\to\infty.
    \label{eq:terminal-geometry-limit}
\end{equation}
In particular, for any two vertices $i,j$ belonging to the same terminal
component,
\begin{equation}
    x_j(t)-x_i(t)
    \longrightarrow
    x_j(t_M)-x_i(t_M)
    +
    \frac{1}{\kappa}
    (e_j-e_i)^TL_{G_\infty}^\#v(t_M).
    \label{eq:terminal-pair-geometry}
\end{equation}
\end{corollary}

\begin{proof}
Apply Proposition~\ref{prop:fixed-trajectory} to the terminal graph with
initial time shifted to $t_M$ and let $s\to\infty$.
\end{proof}

\section{Explicit fragmentation thresholds on a path}
\label{sec:path}

The general theory of Subsections~\ref{sec:fixed}--\ref{sec:prediction}
predicts topology changes recursively for any communication graph arising in
the expansive regime.  For a path graph the structure is substantially more
explicit.  The adjacent velocity differences form a closed symmetric linear
system, its Green matrix can be written in closed form, and finite-time
fragmentation is characterized by a sharp critical alignment strength.

\subsection{Path geometry and adjacent-difference dynamics}

Assume that the initial communication graph is the path
\[
    P_N:\qquad 1-2-\cdots-N.
\]
With
\[
    d_i^0:=x_{i+1}^0-x_i^0,
    \qquad
    i=1,\ldots,N-1,
\]
this is equivalent, under the strict cut-off convention, to
\begin{equation}
    0<d_i^0<r,
    \qquad i=1,\ldots,N-1,
    \label{eq:path-nearest}
\end{equation}
together with
\begin{equation}
    d_i^0+d_{i+1}^0\ge r,
    \qquad i=1,\ldots,N-2.
    \label{eq:path-next-nearest}
\end{equation}
Indeed, \eqref{eq:path-nearest} makes every nearest-neighbor pair active,
whereas \eqref{eq:path-next-nearest} excludes every pair at index distance
two, and therefore every more distant pair as well.

Assume also that the initial velocities are nondecreasing:
\begin{equation}
    v_1^0\le v_2^0\le\cdots\le v_N^0.
    \label{eq:path-ordered-v}
\end{equation}
Set
\[
    m:=N-1,
    \qquad
    q_i:=v_{i+1}-v_i,
    \qquad
    d_i:=x_{i+1}-x_i.
\]
Then $q^0\ge0$.  By Theorem~\ref{thm:order-preservation}, no new edge can
appear.  Hence the graph remains $P_N$ until the first path edge is deleted.

For $N\ge3$, the endpoint velocities satisfy
\[
    \dot v_1=\kappa q_1,
    \qquad
    \dot v_N=-\kappa q_m,
\]
while for $2\le i\le N-1$,
\begin{equation}
    \dot v_i
    =
    \frac{\kappa}{2}(q_i-q_{i-1}).
    \label{eq:path-internal-v}
\end{equation}
Consequently,
\[
    \dot q_1
    =
    \frac{\kappa}{2}(q_2-3q_1),
\]
\[
    \dot q_i
    =
    \frac{\kappa}{2}
    (q_{i-1}-2q_i+q_{i+1}),
    \qquad
    2\le i\le m-1,
\]
and
\[
    \dot q_m
    =
    \frac{\kappa}{2}(q_{m-1}-3q_m).
\]

Define, for $m\ge2$,
\begin{equation}
    C_m
    :=
    \begin{pmatrix}
        3&-1&&&\\
        -1&2&-1&&\\
        &-1&2&\ddots&\\
        &&\ddots&2&-1\\
        &&&-1&3
    \end{pmatrix}
    \in\R^{m\times m},
    \label{eq:Cm}
\end{equation}
and for $m=1$ set
\begin{equation}
    C_1:=[4].
    \label{eq:C1}
\end{equation}
Then the adjacent velocity differences satisfy the unified equation
\begin{equation}
    \dot q
    =
    -\frac{\kappa}{2}C_mq.
    \label{eq:path-q-system}
\end{equation}

\begin{lemma}[Spectral and positivity properties of $C_m$]
\label{lem:Cm-properties}
For every $m\ge1$, $C_m$ is symmetric positive definite.  For $m\ge2$,
its eigenvalues are
\begin{equation}
    \lambda_k
    =
    2-2\cos\frac{k\pi}{m},
    \qquad
    k=1,\ldots,m.
    \label{eq:Cm-eigenvalues}
\end{equation}
Moreover,
\begin{equation}
    e^{-\kappa C_mt/2}\ge0
    \qquad\text{entrywise for every }t\ge0.
    \label{eq:path-positive-semigroup}
\end{equation}
If $m\ge2$, $q^0\ge0$, and $q^0\ne0$, then
\begin{equation}
    q_i(t)>0
    \qquad
    \text{for every }i=1,\ldots,m
    \text{ and every }t>0.
    \label{eq:path-strict-positive-q}
\end{equation}
\end{lemma}

\begin{proof}
The case $m=1$ is immediate.  For $m\ge2$,
\begin{equation}
    z^TC_mz
    =
    \sum_{i=1}^{m-1}(z_i-z_{i+1})^2
    +
    2z_1^2+2z_m^2,
    \label{eq:Cm-quadratic}
\end{equation}
which is strictly positive for $z\ne0$.  Thus $C_m$ is symmetric positive
definite.

The eigenvalue formula \eqref{eq:Cm-eigenvalues} follows by solving the
second-order difference equation in the interior together with the two
endpoint conditions.  Equivalently, one may use the eigenvectors with
components proportional to
\[
    \sin\left(\frac{(i-\frac12)k\pi}{m}\right),
    \qquad i=1,\ldots,m,
\]
with the usual separate normalization for the highest mode $k=m$.

The matrix $-\frac{\kappa}{2}C_m$ is Metzler, so it generates a
nonnegative semigroup, proving \eqref{eq:path-positive-semigroup}.  For
$m\ge2$ the Metzler matrix is irreducible.  Hence its exponential is
strictly positive for every $t>0$, which gives
\eqref{eq:path-strict-positive-q} whenever $q^0\ge0$ is nonzero.
\end{proof}

\begin{remark}[Alignment time scale on a long path]
For $m\gg1$,
\[
    \lambda_1
    =
    2-2\cos\frac{\pi}{m}
    \sim
    \frac{\pi^2}{m^2}.
\]
Thus the slowest decay time in \eqref{eq:path-q-system} scales as
\[
    \frac{2}{\kappa\lambda_1}
    \sim
    \frac{2m^2}{\kappa\pi^2}.
\]
Long paths therefore align increasingly slowly at the global scale.
\end{remark}

\subsection{Exact gap dynamics and the path Green matrix}

Solving \eqref{eq:path-q-system} gives
\begin{equation}
    q(t)
    =
    e^{-\kappa C_mt/2}q^0.
    \label{eq:path-q-solution}
\end{equation}
Since $\dot d=q$,
\begin{equation}
    d(t)
    =
    d^0
    +
    \frac{2}{\kappa}
    C_m^{-1}
    \left(I-e^{-\kappa C_mt/2}\right)q^0.
    \label{eq:path-d-solution}
\end{equation}
If the path is held fixed indefinitely, then
\begin{equation}
    d^\ast
    :=
    \lim_{t\to\infty}d(t)
    =
    d^0+\frac{2}{\kappa}C_m^{-1}q^0.
    \label{eq:path-terminal-gap}
\end{equation}

The inverse of $C_m$ has an explicit Green-kernel representation.

\begin{proposition}[Closed-form path Green matrix]
\label{prop:Cm-inverse}
For every $m\ge1$,
\begin{equation}
    (C_m^{-1})_{ij}
    =
    \frac{
        (2\min\{i,j\}-1)
        \,[\,2(m-\max\{i,j\})+1\,]
    }{4m},
    \qquad
    1\le i,j\le m.
    \label{eq:Cm-inverse}
\end{equation}
In particular,
\begin{equation}
    (C_m^{-1})_{ij}>0
    \qquad
    \text{for all }i,j.
    \label{eq:Cm-inverse-positive}
\end{equation}
\end{proposition}

\begin{proof}
For $m=1$, \eqref{eq:Cm-inverse} gives $C_1^{-1}=[1/4]$.
Assume $m\ge2$ and define $G$ by the right-hand side of
\eqref{eq:Cm-inverse}.  Fix a column $j$.  For $i\le j$,
\[
    G_{ij}
    =
    \frac{(2i-1)[\,2(m-j)+1\,]}{4m},
\]
which is affine in $i$, while for $i\ge j$,
\[
    G_{ij}
    =
    \frac{(2j-1)[\,2(m-i)+1\,]}{4m},
\]
which is also affine in $i$.  Therefore the interior second difference
vanishes away from $i=j$:
\[
    -G_{i-1,j}+2G_{ij}-G_{i+1,j}=0,
    \qquad
    i\ne j.
\]
At $i=j$, the left and right discrete slopes differ by one, giving
\[
    -G_{j-1,j}+2G_{jj}-G_{j+1,j}=1
\]
for an interior index $j$.  The endpoint rows satisfy
\[
    3G_{1j}-G_{2j}=\delta_{1j},
    \qquad
    -G_{m-1,j}+3G_{mj}=\delta_{mj}.
\]
Thus $C_mG=I$, proving \eqref{eq:Cm-inverse}.  Positivity is immediate from
the explicit formula.
\end{proof}

\begin{corollary}[Nonlocal accumulation of velocity gradients]
\label{cor:path-nonlocal}
If $q^0\ge0$ and $q^0\ne0$, then
\begin{equation}
    (C_m^{-1}q^0)_i>0
    \qquad
    \text{for every }i.
    \label{eq:path-nonlocal-positive}
\end{equation}
Thus a positive initial velocity difference anywhere on the path contributes
to the total accumulated expansion of every path edge.
\end{corollary}

\begin{proof}
This follows from \eqref{eq:Cm-inverse-positive}.
\end{proof}

\begin{remark}[A conserved frozen-path vector]
Equation \eqref{eq:path-q-system} and $\dot d=q$ imply
\begin{equation}
    \frac{d}{dt}
    \left(
        d+\frac{2}{\kappa}C_m^{-1}q
    \right)
    =0.
    \label{eq:path-conserved-vector}
\end{equation}
Thus $d^\ast$ in \eqref{eq:path-terminal-gap} is not merely a formal
long-time limit: it is the value of the conserved vector
$d+\frac{2}{\kappa}C_m^{-1}q$ along every fixed-path interval.
\end{remark}

\subsection{Sharp finite-time fragmentation criterion}

Because a path edge is a bridge, the first edge loss disconnects the graph.
The frozen terminal gaps therefore yield a necessary-and-sufficient
fragmentation criterion.

\begin{theorem}[Exact path fragmentation criterion]
\label{thm:path-fragmentation}
Assume \eqref{eq:path-nearest}--\eqref{eq:path-ordered-v}.  Define $d^\ast$
by \eqref{eq:path-terminal-gap}.  Then:

\begin{enumerate}[label=(\roman*)]
\item the path remains connected for every finite time if and only if
\begin{equation}
    d_i^\ast\le r
    \qquad
    \text{for every }i=1,\ldots,m;
    \label{eq:path-no-frag-condition}
\end{equation}

\item finite-time fragmentation occurs if and only if
\begin{equation}
    \max_{1\le i\le m}d_i^\ast>r.
    \label{eq:path-frag-condition}
\end{equation}
\end{enumerate}

If $d_i^\ast=r$ for one or more indices while
$d_j^\ast\le r$ for every $j$, the corresponding critical gaps approach
$r$ only asymptotically and no finite-time fragmentation occurs.
\end{theorem}

\begin{proof}
While the graph remains $P_N$, Lemma~\ref{lem:Cm-properties} gives
$q(t)\ge0$, so each gap $d_i(t)$ is nondecreasing.  If
$d_i^\ast<r$, then $d_i(t)<r$ for every finite $t$.  If
$d_i^\ast=r$, then $d_i(t)$ approaches $r$ monotonically.  It cannot reach
$r$ at a finite time: otherwise monotonicity and the limiting value would
force it to remain identically equal to $r$ thereafter, contradicting the
analytic fixed-path dynamics and the initial inequality $d_i^0<r$.

If some $d_i^\ast>r$, continuity and monotonicity imply that this gap reaches
$r$ in finite time.  The corresponding path edge is then deleted and, being
a bridge, disconnects the graph.  Conversely, every finite-time
fragmentation of a path must begin with the deletion of some path edge, and
Theorem~\ref{thm:frozen-edge-loss} implies that its frozen terminal gap
exceeds $r$.
\end{proof}

For $q^0\ge0$, define the accumulated expansion coefficients
\begin{equation}
    F_i(q^0)
    :=
    2(C_m^{-1}q^0)_i.
    \label{eq:Fi}
\end{equation}
Then
\[
    d_i^\ast=d_i^0+\frac{F_i(q^0)}{\kappa}.
\]

\begin{theorem}[Sharp critical alignment strength]
\label{thm:kappa-critical}
Under the hypotheses of Theorem~\ref{thm:path-fragmentation}, define
\begin{equation}
    \kappa_c
    :=
    \max_{1\le i\le m}
    \frac{
        2(C_m^{-1}q^0)_i
    }{
        r-d_i^0
    }.
    \label{eq:kappa-critical}
\end{equation}
If $q^0=0$, then $\kappa_c=0$.  For every $\kappa>0$,
\begin{equation}
    0<\kappa<\kappa_c
    \quad\Longleftrightarrow\quad
    \text{finite-time fragmentation},
    \label{eq:kappa-below}
\end{equation}
whereas
\begin{equation}
    \kappa\ge\kappa_c
    \quad\Longleftrightarrow\quad
    \text{no finite-time fragmentation}.
    \label{eq:kappa-above}
\end{equation}
If $q^0\ne0$ and $\kappa=\kappa_c$, at least one path gap converges to $r$
as $t\to\infty$, but no edge is deleted at finite time.
\end{theorem}

\begin{proof}
Because $r-d_i^0>0$, the condition $d_i^\ast>r$ is equivalent to
\[
    \kappa
    <
    \frac{2(C_m^{-1}q^0)_i}{r-d_i^0}.
\]
Taking the maximum over $i$ and applying
Theorem~\ref{thm:path-fragmentation} proves
\eqref{eq:kappa-below}--\eqref{eq:kappa-above}.  If $q^0\ne0$, positivity of
$C_m^{-1}$ implies $(C_m^{-1}q^0)_i>0$ for every $i$.  At
$\kappa=\kappa_c$, at least one maximizing edge satisfies
$d_i^\ast=r$, and the equality statement follows from
Theorem~\ref{thm:path-fragmentation}.
\end{proof}

\begin{remark}[Dimensionless form]
Let
\[
    \alpha_i:=\frac{d_i^0}{r},
    \qquad
    \beta_i:=\frac{q_i^0}{\kappa r}.
\]
Then
\begin{equation}
    \frac{d^\ast}{r}
    =
    \alpha+2C_m^{-1}\beta,
    \label{eq:path-dimensionless}
\end{equation}
and fragmentation is equivalent to
\begin{equation}
    \max_i
    \left[
        \alpha_i+2(C_m^{-1}\beta)_i
    \right]
    >1.
    \label{eq:path-dimensionless-frag}
\end{equation}
Thus the path threshold depends only on dimensionless geometry and velocity
differences.
\end{remark}

\subsection{First fragmentation time}

When $\kappa<\kappa_c$, define
\[
    \mathcal I_+
    :=
    \{i:d_i^\ast>r\}.
\]
For every $i\in\mathcal I_+$, let $T_i$ be the unique solution of
\begin{equation}
    d_i^0
    +
    \frac{2}{\kappa}
    e_i^TC_m^{-1}
    \left(I-e^{-\kappa C_mT_i/2}\right)q^0
    =
    r.
    \label{eq:path-edge-hit-time}
\end{equation}
Then
\begin{equation}
    T_{\mathrm{frag}}
    =
    \min_{i\in\mathcal I_+}T_i.
    \label{eq:path-first-frag-time}
\end{equation}
If $q^0\ne0$, Lemma~\ref{lem:Cm-properties} gives $q_i(t)>0$ for all
$t>0$, so every candidate gap is strictly increasing and each $T_i$ is
unique.  Several path edges may attain the minimum simultaneously; all of
them are deleted at the first fragmentation event.

After an edge deletion, every connected component is again a path, the
velocity ordering remains valid, and the same calculation applies
recursively to each component.  Hence the complete terminal cluster
partition of an expansive path can be obtained using only path Green
matrices of smaller sizes.  Since $P_N$ has only $N-1$ edges, at most
$N-1$ edge-deletion events can occur.

\subsection{Two explicit examples}

\begin{example}[Two agents]
\label{ex:path-N2}
For $N=2$, write
\[
    d_0=x_2^0-x_1^0,
    \qquad
    q_0=v_2^0-v_1^0\ge0.
\]
Since $C_1=[4]$,
\[
    q(t)=q_0e^{-2\kappa t}
\]
and
\begin{equation}
    d(t)
    =
    d_0+
    \frac{q_0}{2\kappa}
    \left(1-e^{-2\kappa t}\right).
    \label{eq:N2-d}
\end{equation}
Thus
\[
    d^\ast=d_0+\frac{q_0}{2\kappa},
\]
and
\begin{equation}
    \kappa_c
    =
    \frac{q_0}{2(r-d_0)}.
    \label{eq:N2-kappa-critical}
\end{equation}
For $\kappa<\kappa_c$, the fragmentation time is explicitly
\begin{equation}
    T_{\mathrm{frag}}
    =
    -\frac{1}{2\kappa}
    \log
    \left(
        1-\frac{2\kappa(r-d_0)}{q_0}
    \right).
    \label{eq:N2-frag-time}
\end{equation}
At $\kappa=\kappa_c$, $d(t)\uparrow r$ only as $t\to\infty$.
\end{example}

\begin{example}[Three-agent path]
\label{ex:path-N3}
For $N=3$, let
\[
    a_0=x_2^0-x_1^0,
    \qquad
    b_0=x_3^0-x_2^0,
\]
and
\[
    u_0=v_2^0-v_1^0,
    \qquad
    w_0=v_3^0-v_2^0,
\]
with $u_0,w_0\ge0$.  Here
\[
    C_2
    =
    \begin{pmatrix}
        3&-1\\
        -1&3
    \end{pmatrix},
    \qquad
    C_2^{-1}
    =
    \frac18
    \begin{pmatrix}
        3&1\\
        1&3
    \end{pmatrix}.
\]
Therefore
\begin{equation}
    a^\ast
    =
    a_0+\frac{3u_0+w_0}{4\kappa},
    \qquad
    b^\ast
    =
    b_0+\frac{u_0+3w_0}{4\kappa}.
    \label{eq:N3-terminal-gaps}
\end{equation}
Finite-time fragmentation occurs if and only if at least one of these
quantities exceeds $r$, and the sharp critical coupling is
\begin{equation}
    \kappa_c
    =
    \max
    \left\{
        \frac{3u_0+w_0}{4(r-a_0)},
        \frac{u_0+3w_0}{4(r-b_0)}
    \right\}.
    \label{eq:N3-kappa-critical}
\end{equation}
The cross terms in \eqref{eq:N3-terminal-gaps} illustrate the nonlocal
nature of the path Green matrix: the initial velocity difference on either
edge contributes to the accumulated expansion of the other.
\end{example}

\section{Extensions and limits of the finite-size theory}
\label{sec:spectral}

We now delimit and extend the hypotheses behind the main results. First, the
ordered-velocity assumption can be relaxed for a class of path data: the
slowest fixed-path mode may drive mixed-sign velocity differences into the
expansive cone before any geometric event. We then isolate the finite-event
mechanism as a conditional principle for switching linear relaxation and
verify it for a family of normalized averages with uniform self-weight.

The distinction between \emph{spectral entry} and the fully general switching
problem is essential.  The spectral calculation below is exact while the
communication graph remains a path; additional geometric conditions are
needed to guarantee that the actual state-dependent system remains in that
path cell long enough for the entry to occur.

\subsection{Eventual positivity for the frozen path}

Let $N\ge3$, $m=N-1$, and suppose for the moment that the communication graph
is held fixed at $P_N$.  We allow arbitrary
\[
    q^0=(v_2^0-v_1^0,\ldots,v_N^0-v_{N-1}^0)^T\in\R^m,
\]
with no sign restriction.  By \eqref{eq:path-q-system},
\begin{equation}
    q(t)=e^{-\kappa C_mt/2}q^0.
    \label{eq:spectral-q}
\end{equation}

The smallest eigenvalue of $C_m$ is
\[
    \lambda_1
    =
    2-2\cos\frac{\pi}{m},
\]
and a normalized associated eigenvector is
\begin{equation}
    \phi_{1,i}
    =
    \sqrt{\frac{2}{m}}
    \sin\left(
        \frac{(i-\frac12)\pi}{m}
    \right),
    \qquad
    i=1,\ldots,m.
    \label{eq:first-eigenvector}
\end{equation}
Every component of $\phi_1$ is strictly positive.  Define
\begin{equation}
    c_1:=\phi_1^Tq^0.
    \label{eq:c1}
\end{equation}

\begin{theorem}[Spectral characterization of eventual expansive entry]
\label{thm:spectral-eventual-positive}
Assume $q^0\ne0$ and let the path be held fixed.  Then the following are
equivalent:
\begin{enumerate}[label=(\roman*)]
\item $c_1>0$;
\item there exists a finite time $T$ such that $q(t)\ge0$ for all
$t\ge T$;
\item there exists a finite time $T$ such that
\begin{equation}
    q_i(t)>0
    \qquad
    \text{for every }i=1,\ldots,m
    \text{ and all }t>T.
    \label{eq:eventual-strict-positive}
\end{equation}
\end{enumerate}
\end{theorem}

\begin{proof}
Let $\{\phi_k\}_{k=1}^m$ be an orthonormal eigenbasis of $C_m$ and write
\[
    q^0
    =
    c_1\phi_1+\sum_{k=2}^m c_k\phi_k.
\]
Then
\begin{equation}
    q(t)
    =
    c_1e^{-\kappa\lambda_1t/2}\phi_1
    +
    \sum_{k=2}^m
    c_ke^{-\kappa\lambda_kt/2}\phi_k.
    \label{eq:spectral-decomposition-q}
\end{equation}
Because $\lambda_1<\lambda_k$ for every $k\ge2$, if $c_1>0$ the first term
eventually dominates all higher modes.  Since $\phi_1$ is strictly positive
componentwise, $q(t)$ is strictly positive componentwise for all sufficiently
large $t$.  Thus (i) implies (iii), and (iii) implies (ii).

Conversely,
\[
    \phi_1^Tq(t)
    =
    c_1e^{-\kappa\lambda_1t/2}.
    \label{eq:first-mode-projection}
\]
If $q(t)\ge0$ at any finite time, then $q(t)\ne0$ because the matrix
exponential in \eqref{eq:spectral-q} is invertible and $q^0\ne0$.  Since
$\phi_1>0$ componentwise,
\[
    \phi_1^Tq(t)>0,
\]
and hence $c_1>0$.  Therefore (ii) implies (i).
\end{proof}

\begin{remark}[The cases $c_1\le0$]
If $c_1<0$, the slowest mode in
\eqref{eq:spectral-decomposition-q} eventually dominates with negative
sign, so $q_i(t)<0$ for every $i$ for all sufficiently large $t$.  If
$c_1=0$ and $q^0\ne0$, then
$\phi_1^Tq(t)=0$ for every $t$.  Because $\phi_1$ is strictly positive,
$q(t)$ can never lie in the nonnegative orthant unless $q(t)=0$, which is
impossible at finite time for nonzero $q^0$.  Thus $c_1>0$ is not merely a
sufficient slow-mode condition; it is the exact fixed-path criterion for
eventual entry into the expansive velocity cone.
\end{remark}

\subsection{An explicit ordering-time bound}

The proof above gives a quantitative sufficient time for entry.  Define
\begin{equation}
    r^0:=q^0-c_1\phi_1,
    \qquad
    R:=\|r^0\|_2,
    \label{eq:spectral-remainder}
\end{equation}
and
\begin{equation}
    \phi_\ast
    :=
    \min_{1\le i\le m}\phi_{1,i}
    =
    \sqrt{\frac{2}{m}}
    \sin\frac{\pi}{2m}.
    \label{eq:phi-star}
\end{equation}
If $c_1>0$ and $R>0$, set
\begin{equation}
    T_{\rm ord}
    :=
    \frac{2}{\kappa(\lambda_2-\lambda_1)}
    \left[
        \log\frac{R}{c_1\phi_\ast}
    \right]_+,
    \qquad
    [a]_+:=\max\{a,0\}.
    \label{eq:Tord}
\end{equation}
If $R=0$, set $T_{\rm ord}:=0$.

\begin{proposition}[Explicit entry-time estimate]
\label{prop:Tord}
If $c_1>0$, then under the frozen path dynamics
\begin{equation}
    q(t)\ge0
    \qquad
    \text{for every }t\ge T_{\rm ord},
    \label{eq:Tord-nonnegative}
\end{equation}
and $q(t)>0$ componentwise for every $t>T_{\rm ord}$.
\end{proposition}

\begin{proof}
The higher-mode remainder
\[
    r(t)
    :=
    \sum_{k=2}^m
    c_ke^{-\kappa\lambda_kt/2}\phi_k
\]
satisfies
\[
    \|r(t)\|_2
    \le
    R e^{-\kappa\lambda_2t/2}.
\]
For every coordinate,
\[
\begin{aligned}
    q_i(t)
    &=
    c_1e^{-\kappa\lambda_1t/2}\phi_{1,i}
    +r_i(t)\\
    &\ge
    c_1\phi_\ast e^{-\kappa\lambda_1t/2}
    -
    R e^{-\kappa\lambda_2t/2}.
\end{aligned}
\]
The right-hand side is nonnegative whenever
\[
    e^{-\kappa(\lambda_2-\lambda_1)t/2}
    \le
    \frac{c_1\phi_\ast}{R},
\]
which is exactly the condition encoded in \eqref{eq:Tord}.  For
$t>T_{\rm ord}$ the inequality is strict.
\end{proof}

\subsection{Entry before a geometric or topology event}

The preceding calculation is exact only while the graph remains $P_N$ and
the particle labels retain their spatial order.  For mixed-sign $q^0$, a
gap may initially shrink, so before spectral entry one must exclude three
possibilities: particle crossing, deletion of a path edge, and creation of a
next-nearest-neighbor edge.

For this purpose assume the initial configuration lies in the \emph{strict
path cell}
\begin{equation}
    0<d_i^0<r,
    \qquad i=1,\ldots,m,
    \label{eq:strict-path-nearest}
\end{equation}
and
\begin{equation}
    d_i^0+d_{i+1}^0>r,
    \qquad i=1,\ldots,m-1.
    \label{eq:strict-path-next}
\end{equation}
Define the geometric safety margin
\begin{equation}
    \gamma_0
    :=
    \min
    \left\{
        \min_i d_i^0,\,
        \min_i(r-d_i^0),\,
        \min_i(d_i^0+d_{i+1}^0-r)
    \right\}.
    \label{eq:gamma0}
\end{equation}
Thus $\gamma_0>0$.  The three terms measure the initial distance to,
respectively, particle crossing, path-edge deletion, and next-nearest edge
creation.

Let
\begin{equation}
    D_0
    :=
    \max_i v_i^0-\min_i v_i^0
    \label{eq:D0-spectral}
\end{equation}
be the initial velocity diameter.

\begin{theorem}[A checkable safe-entry condition]
\label{thm:safe-entry}
Assume \eqref{eq:strict-path-nearest}--\eqref{eq:strict-path-next},
$q^0\ne0$, and $c_1>0$.  If
\begin{equation}
    D_0T_{\rm ord}<\gamma_0,
    \label{eq:safe-entry-condition}
\end{equation}
then the actual state-dependent dynamics remains in the strict path cell
until time $T_{\rm ord}$.  In particular,
\[
    G(t)=P_N
    \qquad
    \text{for }0\le t\le T_{\rm ord},
\]
and
\begin{equation}
    q(T_{\rm ord})\ge0.
    \label{eq:q-at-entry}
\end{equation}
Consequently, from time $T_{\rm ord}$ onward the actual trajectory lies in
the expansive regime of Theorem~\ref{thm:order-preservation}.
\end{theorem}

\begin{proof}
Consider first the frozen-path solution.  Proposition
\ref{prop:velocity-hull} applies to this fixed graph, so its velocity
diameter is bounded by $D_0$.  Hence for every pair $i<j$,
\begin{equation}
    \left|
        [x_j(t)-x_i(t)]
        -
        [x_j^0-x_i^0]
    \right|
    \le D_0t.
    \label{eq:pair-displacement-bound}
\end{equation}
In particular, for $0\le t\le T_{\rm ord}$,
condition \eqref{eq:safe-entry-condition} prevents every adjacent gap from
reaching either $0$ or $r$, and prevents every next-nearest separation
$d_i+d_{i+1}$ from reaching $r$.  No more distant pair can create an edge
before a next-nearest pair does while the spatial ordering remains strict.
Thus the frozen path remains inside the strict path cell through
$T_{\rm ord}$.

Since no boundary of the path cell is reached before $T_{\rm ord}$, the
frozen trajectory is exactly the actual state-dependent trajectory on that
time interval.  Proposition~\ref{prop:Tord} gives
\eqref{eq:q-at-entry}.  The state at $T_{\rm ord}$ therefore satisfies the
velocity ordering required by Theorem~\ref{thm:order-preservation}, which
preserves that ordering thereafter.
\end{proof}

Condition \eqref{eq:safe-entry-condition} is deliberately conservative.  A
sharper trajectory-based sufficient condition can be stated directly from
the frozen trajectory.  Define
\[
    \widehat q(t)
    :=
    e^{-\kappa C_mt/2}q^0
\]
and
\begin{equation}
    \widehat d(t)
    :=
    d^0+
    \frac{2}{\kappa}
    C_m^{-1}
    \left(I-e^{-\kappa C_mt/2}\right)q^0.
    \label{eq:frozen-mixed-d}
\end{equation}
When $c_1>0$, let
\begin{equation}
    \tau_+
    :=
    \inf
    \left\{
        T\ge0:
        \widehat q(t)\ge0
        \text{ for every }t\ge T
    \right\},
    \label{eq:tau-plus}
\end{equation}
which is finite by
Theorem~\ref{thm:spectral-eventual-positive}.  Let $\tau_{\rm exit}$ be the
first time at which the frozen trajectory reaches the boundary of the strict
path cell:
\begin{equation}
\begin{split}
    \tau_{\rm exit}
    :=
    \inf\{t>0:\;&
        \widehat d_i(t)=0
        \text{ for some }i,\ \text{or}\\
        &\widehat d_i(t)=r
        \text{ for some }i,\ \text{or}\\
        &\widehat d_i(t)+\widehat d_{i+1}(t)=r
        \text{ for some }i
    \}.
    \label{eq:tau-exit}
\end{split}
\end{equation}
As usual, the infimum of the empty set is $+\infty$.

\begin{proposition}[Trajectory-based safe-entry condition]
\label{prop:exact-safe-entry}
If
\begin{equation}
    c_1>0
    \qquad\text{and}\qquad
    \tau_+<\tau_{\rm exit},
    \label{eq:exact-entry-condition}
\end{equation}
then the actual trajectory coincides with the frozen path through time
$\tau_+$, enters the expansive cone at $\tau_+$, and thereafter evolves
according to the irreversible edge-deletion theory of
Section~\ref{sec:order}, beginning with the current state at $\tau_+$.
\end{proposition}

\begin{proof}
By definition of $\tau_{\rm exit}$, the frozen path remains strictly inside
the path cell on $[0,\tau_+]$.  Hence no state-dependent topology change or
particle crossing can distinguish the actual trajectory from the frozen
trajectory before $\tau_+$.  At that time
$\widehat q(\tau_+)\ge0$.  Theorem~\ref{thm:order-preservation} then applies
to the actual trajectory from $\tau_+$ onward.
\end{proof}

\subsection{Fragmentation after spectral entry}

The frozen-path vector
\begin{equation}
    d^\ast
    =
    d^0+\frac{2}{\kappa}C_m^{-1}q^0
    \label{eq:mixed-d-star}
\end{equation}
remains useful even when $q^0$ has mixed signs.  Indeed,
\eqref{eq:path-conserved-vector} implies that as long as the graph is a path,
\begin{equation}
    d(t)+\frac{2}{\kappa}C_m^{-1}q(t)
    =
    d^\ast.
    \label{eq:mixed-conserved-vector}
\end{equation}
Thus, if the system safely reaches the expansive cone before leaving the
path cell, the same $d^\ast$ computed from the original mixed initial data
becomes the terminal-gap predictor for the subsequent expansive path
dynamics.

\begin{theorem}[Exact fragmentation criterion after safe spectral entry]
\label{thm:mixed-fragmentation}
Assume either the sufficient hypotheses of
Theorem~\ref{thm:safe-entry} or the trajectory-based entry condition
\eqref{eq:exact-entry-condition}.  Then
\begin{equation}
    \max_i d_i^\ast>r
    \quad\Longleftrightarrow\quad
    \text{finite-time fragmentation},
    \label{eq:mixed-frag-iff}
\end{equation}
where $d^\ast$ is given by \eqref{eq:mixed-d-star}. If
$d_i^\ast\le r$ for every $i$, there is no finite-time fragmentation. In
that case, any equality $d_i^\ast=r$ corresponds only to asymptotic contact
with the interaction boundary.
\end{theorem}

\begin{proof}
Let $T$ denote a safe entry time, either $T_{\rm ord}$ or $\tau_+$.  At time
$T$ the graph is still $P_N$ and $q(T)\ge0$.  By
\eqref{eq:mixed-conserved-vector},
\[
    d^\ast
    =
    d(T)+\frac{2}{\kappa}C_m^{-1}q(T).
\]
Theorem~\ref{thm:path-fragmentation}, applied with time $T$ as the new
initial time, gives exactly \eqref{eq:mixed-frag-iff} and the equality
statement.
\end{proof}

The explicit bound \eqref{eq:Tord} can also be interpreted as a minimum
alignment rate needed to guarantee entry before the initial safety margin is
exhausted.  When $c_1>0$, define
\begin{equation}
    \kappa_{\rm ent}
    :=
    \frac{2D_0}
    {\gamma_0(\lambda_2-\lambda_1)}
    \left[
        \log\frac{R}{c_1\phi_\ast}
    \right]_+,
    \label{eq:kappa-entry}
\end{equation}
with $\kappa_{\rm ent}=0$ if $R=0$.  Then
\begin{equation}
    \kappa>\kappa_{\rm ent}
    \label{eq:kappa-entry-condition}
\end{equation}
implies the sufficient safe-entry condition
\eqref{eq:safe-entry-condition}.

For mixed $q^0$, the accumulated displacement coefficients
$2(C_m^{-1}q^0)_i$ need not all be positive.  Define
\begin{equation}
    \kappa_{\rm frag}^{\rm mix}
    :=
    \max_i
    \frac{
        [\,2(C_m^{-1}q^0)_i\,]_+
    }{
        r-d_i^0
    }.
    \label{eq:kappa-frag-mixed}
\end{equation}
Within the safely entering regime, finite-time fragmentation occurs for
$\kappa<\kappa_{\rm frag}^{\rm mix}$ and does not occur for
$\kappa\ge\kappa_{\rm frag}^{\rm mix}$.  Unlike the critical coupling in
Theorem~\ref{thm:kappa-critical}, however,
$\kappa_{\rm frag}^{\rm mix}$ is \emph{not} a global critical coupling for
arbitrary mixed initial velocities, because the theorem additionally
requires entry before a path-cell exit.  In particular, the parameter range
\[
    \kappa\le\kappa_{\rm ent}
\]
is not certified by the explicit estimate
$\kappa>\kappa_{\rm ent}$ alone.  The sharper condition
$\tau_+<\tau_{\rm exit}$ may nevertheless still hold there, in which case
the exact safe-entry theory remains applicable.

\begin{example}[An open set of non-monotone data with safe entry]
\label{ex:safe-entry}
Take $N=3$, $r=\kappa=1$, $d^0=(0.7,0.7)$, and
$v^0=(0,-0.01,0.59)$, so $q^0=(-0.01,0.60)$. Here
$\phi_1=(1,1)^T/\sqrt2$, $\lambda_1=2$, $\lambda_2=4$,
$c_1=0.59/\sqrt2$, $R=0.61/\sqrt2$, $D_0=0.60$, and $\gamma_0=0.30$.
Consequently
\[
 T_{\rm ord}=\log\frac{0.61\sqrt2}{0.59}<0.380,
 \qquad D_0T_{\rm ord}<0.228<\gamma_0.
\]
Theorem~\ref{thm:safe-entry} therefore certifies entry without a prior
geometric event. Nevertheless,
\[
 d^*=\left(0.7+\frac{3(-0.01)+0.60}{4},\,
           0.7+\frac{-0.01+3(0.60)}{4}\right)
     =(0.8425,1.1475),
\]
so Theorem~\ref{thm:mixed-fragmentation} certifies subsequent finite-time
fragmentation. All the safety, spectral, mixed-sign, and exceedance
inequalities are strict. Their continuous dependence on the initial data
therefore gives an open neighborhood of non-monotone configurations with
the same certified conclusions.
\end{example}

\subsection{A finite-event principle for integrated relaxation}
\label{sec:mechanism}

The ordered-regime argument uses a general feature of switching relaxation
systems. The following formulation separates that feature from the
one-dimensional neighborhood comparison. It is a sufficient structural
criterion; checking its invariant-cone hypothesis is a model-dependent step.

Let $\mathcal A$ be a finite set of possible interactions. For each
$a\in\mathcal A$, fix a vector $b_a\in\R^n$ and threshold $r_a\in\R$,
and let $h_a(x)=b_a^Tx$. The mode is the active set
$E(x)=\{a:h_a(x)<r_a\}$. In mode $E$, consider
\begin{equation}
 \dot x=v,\qquad \dot v=-\kappa L_Ev,
 \label{eq:abstract-dynamics}
\end{equation}
where $L_E$ is constant, and keep $x,v$ continuous at each event.
The admissible modes are those compatible with the state region under
consideration, including the modes reached at its switching boundaries.

\begin{theorem}[Monotone switching driven by integrated relaxation]
\label{thm:monotone-relaxation}
Suppose there is a closed cone $\mathcal K\subseteq\R^n$ such that, for
every admissible mode $E$:
\begin{enumerate}[label=(\roman*)]
\item zero, if present in the spectrum of $L_E$, is semisimple, and every
nonzero eigenvalue has positive real part; write
$\Pi_E=\lim_{t\to\infty}e^{-\kappa L_Et}$;
\item $e^{-\kappa L_Et}\mathcal K\subseteq\mathcal K$ for $t\ge0$, and
$b_a^Tw\ge0$ for every $w\in\mathcal K$ and every $a\in\mathcal A$;
\item $b_a^T\Pi_E=0$ for every active $a\in E$.
\end{enumerate}
Assume $v^0\in\mathcal K$ and that the construction remains in the stated
admissible state region. Then the active sets decrease and at most
$|E(x^0)|$ events occur. At a current state $(x,v,E)$, set
\begin{equation}
 S_a^E=b_a^Tx+\kappa^{-1}b_a^TL_E^\#v,
 \qquad a\in E.
 \label{eq:abstract-terminal}
\end{equation}
The mode is terminal if and only if $S_a^E\le r_a$ for every $a\in E$.
Otherwise the next event is the smallest positive solution, over active
$a$ with $S_a^E>r_a$, of
\begin{equation}
 b_a^Tx+\kappa^{-1}b_a^TL_E^\#
       (I-e^{-\kappa L_Et})v=r_a.
 \label{eq:abstract-hit}
\end{equation}
Each such solution is unique. All minimizers are removed simultaneously,
and iteration gives the exact terminal mode and limiting velocity
$\Pi_{E_\infty}v(t_M)$. Equality $S_a^E=r_a$ alone never causes a finite
frozen event.
\end{theorem}

\begin{proof}
Assumption (i) implies exponential decay on the complementary spectral
subspace, including when positive eigenvalues have Jordan blocks. Hence
\[
 L_E^\#=\int_0^\infty(e^{-L_Es}-\Pi_E)\,ds,
 \qquad
 x(t)=x+t\Pi_Ev+
 \kappa^{-1}L_E^\#(I-e^{-\kappa L_Et})v.
\]
The cone is preserved in each mode and at continuous event updates.
Assumption (ii) therefore makes every $h_a(x(t))$ nondecreasing, so an
inactive interaction cannot return. For an active interaction, (iii)
removes the ballistic term and gives the finite limit
\eqref{eq:abstract-terminal}. Its frozen trajectory is analytic and
nondecreasing, starting strictly below $r_a$. If the limit equals $r_a$,
a finite hit would force constancy on a later interval, hence everywhere,
a contradiction. If the limit exceeds $r_a$, continuity gives a hit;
monotonicity and analyticity give uniqueness. The earliest hit is the first
possible mode change, so the actual and frozen solutions coincide until
then. Updating all minimizers and repeating proves exactness. Each event
removes at least one interaction permanently, which bounds the number of
events and precludes their accumulation. The finite sequence of linear
flows extends globally and its last velocity flow converges to the stated
projection.
\end{proof}

For the alignment system, take $b_{ij}=e_j-e_i$ for $i<j$,
$r_{ij}=r$, and $\mathcal K=\{v:Qv\ge0\}$. Strict position order is
preserved, so these signed observables represent the actual distances.
Lemma~\ref{lem:metzler} verifies (ii), while the fixed-graph spectral
analysis verifies (i) and (iii). Thus the special geometric work is the
verification of a common cone; the group-inverse integration and event
selection do not depend on the uniform self-excluding average. The theorem
requires no reversibility of $L_E$, although applying it to a directed
model would require checking all three hypotheses and would not by itself
identify graph components with velocity clusters.

\subsection{A concrete extension: uniform self-weight}

Replace the local velocity equation, for a non-isolated agent, by
\begin{equation}
 \dot v_i=\kappa\left(
 \frac{\sum_{j\in\Nset_i}v_j+\eta v_i}{n_i+\eta}-v_i\right)
 =\frac{\kappa}{n_i+\eta}\sum_{j\in\Nset_i}(v_j-v_i),
 \qquad \eta\ge0,
 \label{eq:self-weight-model}
\end{equation}
with the same $\eta$ for every agent and $\dot v_i=0$ when $n_i=0$.
Here $\eta=0$ gives the original model, and $\eta=1$ gives the arithmetic
average over the neighborhood including the center. Except on a regular
graph, changing $\eta$ is not a common rescaling of time.

\begin{proposition}[Persistence of the reduction under uniform self-weight]
\label{prop:self-weight}
For every fixed $\eta\ge0$, ordered positions and nondecreasing velocities
remain ordered under \eqref{eq:self-weight-model}. All conclusions of the
finite-event recursion hold with $L_G$ replaced on nontrivial components by
\begin{equation}
 L_{G,\eta}=(D_G+\eta I)^{-1}(D_G-A_G),
 \qquad
 \pi_{i,\eta}^{C}
 =\frac{\deg_C(i)+\eta}{\sum_{j\in C}(\deg_C(j)+\eta)},
 \label{eq:self-weight-laplacian}
\end{equation}
and with zero singleton blocks and singleton projection $1$.
In particular the terminal velocity of a nontrivial component is
$\sum_{i\in C}\pi_{i,\eta}^C v_i(t_M)$.
\end{proposition}

\begin{proof}
At a boundary $v_i=v_{i+1}=c$ of the velocity cone, first suppose the two
adjacent agents are not neighbors. All neighbor velocities of $i$ are at
most $c$, and those of $i+1$ are at least $c$, so their accelerations have
the required opposite signs. If they are neighbors, each weighted average
in \eqref{eq:self-weight-model} contains a common mass $1+\eta$ at value
$c$: one unit comes from the other agent and $\eta$ from the center. The
remaining entries have unit mass. Sliding from $I_i$ to $I_{i+1}$ removes
only smallest values on the left and adds only largest values on the right,
as in Lemma~\ref{lem:boundary-acceleration}. Each operation increases or
preserves the weighted mean, and the common mass keeps the denominator
positive. Thus $\dot v_{i+1}-\dot v_i\ge0$. The step-vector argument of
Lemma~\ref{lem:metzler} and the switching induction of
Theorem~\ref{thm:order-preservation} establish cone invariance and
irreversible deletion. The velocity convex-hull bound also holds.

On a nontrivial connected component set $W=D_C+\eta I$. Then
\[
 W^{1/2}L_{C,\eta}W^{-1/2}
   =W^{-1/2}(D_C-A_C)W^{-1/2}
\]
is symmetric positive semidefinite with a one-dimensional nullspace. Its
stationary weights are those in \eqref{eq:self-weight-laplacian}. The
projection is constant on each component, so active pair differences
annihilate it. Theorem~\ref{thm:monotone-relaxation} now gives the terminal
test and finite recursion, and the stationary projection gives the stated
velocities.
\end{proof}

The extension verifies a family of averaging rules, not arbitrary
heterogeneous or distance-dependent weights. The explicit path matrix and
thresholds of Section~\ref{sec:path} use the original convention $\eta=0$.

\section{Numerical validation}
\label{sec:numerics}

The computations test the exact deterministic predictions derived above.

\subsection{Numerical validation protocol}

The computations below compare two implementations.  The theoretical
implementation evaluates the fixed-path formula
\eqref{eq:path-d-solution}, locates its first boundary hit by bisection,
and, for the cascade example, applies the exact event recursion of
Section~\ref{sec:prediction}.  The comparison implementation integrates the
original agent equations \eqref{eq:x-dynamics}--\eqref{eq:isolated} by the
classical fourth-order Runge--Kutta method on each fixed-graph interval.  It
detects an active edge reaching the strict cut-off boundary, locates that
event by bisection, deletes the edge, and restarts the integration with the
updated graph.  Thus the comparison solver does not use the group-inverse
terminal-separation test to decide whether or when an edge is lost.

For the random-path and cascade experiments the maximum time step is
$2\times10^{-4}$, and an event is localized until the bracketing interval
has relative width at most $10^{-12}$.  For the longer threshold sweep we
use $\kappa\Delta t\le2\times10^{-3}$ and integrate to
$T_{\max}=40/\kappa$.  The random generator seed is $20260905$.  The script
\texttt{generate\_numerical\_validation.py} and the three accompanying CSV
files contain the complete parameter choices and reported values.

The threshold test varies $\kappa$ across the value $\kappa_c$ from
\eqref{eq:kappa-critical}.  The critical value itself is deliberately
omitted from the finite-time sweep: under the strict cut-off convention the
critical edge approaches $r$ only asymptotically.  The graph component count
at $T_{\max}$ is therefore reported as $K(T_{\max})$, rather than being
silently identified with a numerically inferred infinite-time limit.

For the first-event test we generate $20$ samples for each
$N\in\{3,4,5\}$.  With $r=1$, the adjacent gaps are sampled independently
from $\operatorname{Unif}[0.45,0.85]$ and rejected unless every sum of two
consecutive gaps exceeds $1.02$; hence the initial graph is a path.  The
relative velocities are sampled independently from
$\operatorname{Unif}[0.08,1]$, and
$\kappa/\kappa_c\sim\operatorname{Unif}[0.4,0.85]$.  All these samples are
strictly expansive and fragment in finite time.  We measure the discrepancy
by
\[
    \frac{|T_{\rm sim}-T_{\rm pred}|}
         {\max\{1,T_{\rm pred}\}}.
\]

The path experiments test the formulas and their implementation. The
additional checks in Section~\ref{sec:extension-checks} use a non-path
interval graph and the certified non-monotone example; all remain within
the proved hypotheses.

\subsection{Illustrative computations}
\label{sec:illustrative}

We first give three concrete comparisons between the theoretical predictor
and the direct switching solver described above.

\subsubsection{Threshold validation}
For the ordered three-agent path with $r=1$, $d_1^0=d_2^0=0.5$,
$q_1^0=0.8$, $q_2^0=0.2$, the critical coupling
\eqref{eq:kappa-critical} evaluates to $\kappa_c=1.3$.
Figure~\ref{fig:kappa-threshold} sweeps $\kappa/\kappa_c$ across this value
and records $K(T_{\max})$.  The complete three-cluster regime at small
$\kappa$, as well as the two-to-one cluster transition nearest the critical
value, is visible.  The closest sampled ratios below and above the predicted
threshold are $0.992$ and $1.008$, where the observed component counts are
$2$ and $1$, respectively.  This bracket is set only by the sampling grid.

\begin{figure}[htbp]
\centering
\input{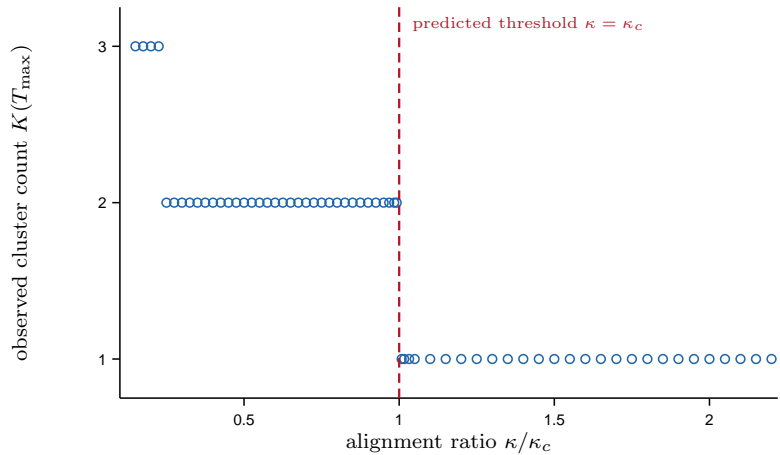}
\caption{Observed component count $K(T_{\max})$, with
$T_{\max}=40/\kappa$, for the three-agent path example.  The dashed line is
the predicted critical value $\kappa=\kappa_c$ from
Theorem~\ref{thm:kappa-critical}; the critical value itself is not sampled}
\label{fig:kappa-threshold}
\end{figure}

\subsubsection{First fragmentation time}
For $60$ randomly generated expansive path configurations with
$N\in\{3,4,5\}$ agents (uniformly sampled subject to
\eqref{eq:path-nearest}--\eqref{eq:path-next-nearest} and $\kappa<\kappa_c$),
Figure~\ref{fig:frag-time} compares the first fragmentation time
$T_{\mathrm{frag}}$ of \eqref{eq:path-first-frag-time} against the time of
the first edge loss in the direct integration.  The largest normalized
discrepancy over the $60$ instances is $7.30\times10^{-13}$.

\begin{figure}[htbp]
\centering
\input{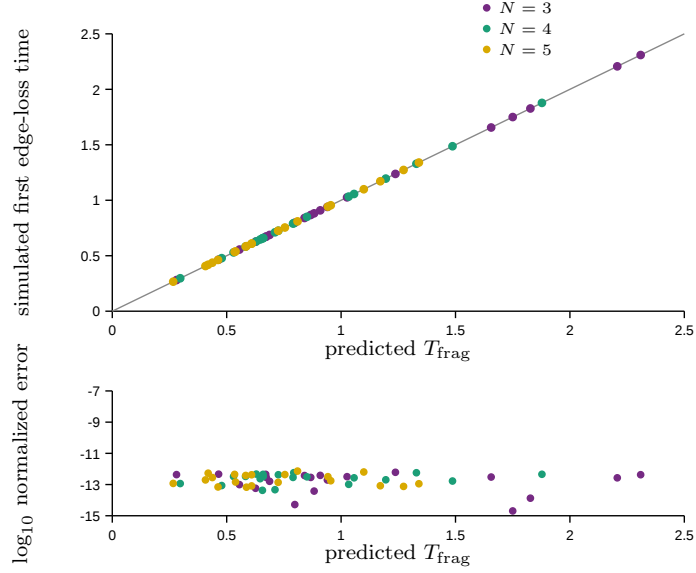}
\caption{Predicted versus directly simulated first fragmentation time for
$60$ expansive path configurations (top), and the base-$10$ logarithm of
the normalized discrepancy (bottom).  Colors distinguish $N=3,4,5$}
\label{fig:frag-time}
\end{figure}

\subsubsection{Recursive terminal-cluster prediction}
Finally we test the complete finite-event recursion of
Section~\ref{sec:prediction} on a four-agent example undergoing three
cascading edge deletions.  With $r=1$, $\kappa=0.5$,
$x^0=(0,\,0.6,\,1.1,\,1.9)$, and $v^0=(0,\,0.3,\,0.7,\,1.6)$, the recursion
\eqref{eq:recursive-terminal-separation}--\eqref{eq:recursive-graph-update}
predicts complete fragmentation into four singleton clusters through three
successive edge-loss events. Figure~\ref{fig:cascade} shows the simulated
particle trajectories together with the predicted event times, and
Table~\ref{tab:cascade} compares the predicted and simulated event times and
terminal velocities.  All event-time discrepancies are below
$6.5\times10^{-13}$, and all terminal-velocity discrepancies are below
$1.5\times10^{-13}$; the latter are at the level of accumulated floating-point
round-off and are therefore reported only to one significant figure.

\begin{figure}[htbp]
\centering
\input{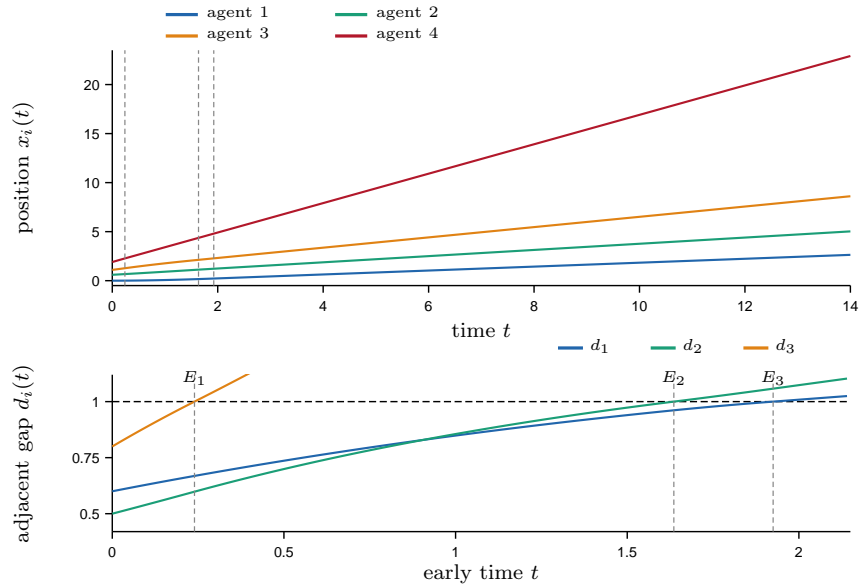}
\caption{Directly simulated positions for the four-agent cascade (top) and
the three adjacent gaps near the topology changes (bottom).  Dashed vertical
lines mark the event times predicted by the recursion of
Section~\ref{sec:prediction}, and the horizontal dashed line is the cut-off
$d_i=r=1$.  Positions and velocities remain continuous at each event; the
lower panel makes the three boundary crossings visible}
\label{fig:cascade}
\end{figure}

\begin{table}[htbp]
\centering
\caption{Recursive prediction versus direct simulation for the cascading
four-agent example of Figure~\ref{fig:cascade}. Event times are measured
from $t=0$; terminal velocities are the asymptotic values
$v_i(t\to\infty)$.}
\label{tab:cascade}
\small
\begin{tabular}{@{}lccc@{}}
\toprule
Event & Predicted time & Simulated time & $|\text{difference}|$ \\
\midrule
1 (edge $(3,4)$ lost) & $0.239433834830$ & $0.239433834830$ & $3.66\times10^{-13}$ \\
2 (edge $(2,3)$ lost) & $1.635738173075$ & $1.635738173076$ & $6.49\times10^{-13}$ \\
3 (edge $(1,2)$ lost) & $1.924959331603$ & $1.924959331604$ & $4.65\times10^{-13}$ \\
\addlinespace
Agent & Predicted $v_i^\infty$ & Simulated $v_i^\infty$ & $|\text{difference}|$ \\
\midrule
$1$ & $0.200000000000$ & $0.200000000000$ & $4\times10^{-14}$ \\
$2$ & $0.315329957109$ & $0.315329957109$ & $5\times10^{-15}$ \\
$3$ & $0.524539662172$ & $0.524539662172$ & $1\times10^{-13}$ \\
$4$ & $1.500000000000$ & $1.500000000000$ & $1\times10^{-13}$ \\
\bottomrule
\end{tabular}
\end{table}

The agreement is consistent with the chosen integration and event-location
tolerances.  These three comparisons check the path theory and its event recursion.

\FloatBarrier
\subsection{Non-path and self-weight checks}
\label{sec:extension-checks}

To test the recursion beyond paths, take $r=1$, $\kappa=0.5$,
$x^0=(0,0.35,0.70,1.30)$, and $v^0=(0,0.15,0.80,1.60)$.
The initial graph contains all pairs except $(1,4)$. For both $\eta=0$ and
$\eta=1$, the predictor removes $(2,4)$, $(1,3)$, $(3,4)$, and $(2,3)$ in
that order. The first two deletions leave the graph connected; the third
causes fragmentation. Both terminal partitions are
$\{\{1,2\},\{3\},\{4\}\}$, but the event times and limiting velocities
differ (Table~\ref{tab:self-weight}). This example distinguishes loss of a
non-bridge edge from fragmentation and tests recomputation of the
normalization at successive events.

\begin{table}[htbp]
\centering
\caption{Non-path recursion for two self-weights. $T_{\rm loss}$ is the first
edge-loss time, $T_{\rm frag}$ is the first disconnection time, and
$V_{\{1,2\}}^\infty$ is the limiting velocity of the two-agent terminal
component. Values shown are the predictor outputs.}
\label{tab:self-weight}
\begin{tabular}{@{}cccc@{}}
\toprule
$\eta$ & $T_{\rm loss}$ & $T_{\rm frag}$ & $V_{\{1,2\}}^\infty$\\
\midrule
$0$ & $0.0348517232$ & $0.5583514317$ & $0.2117880084$\\
$1$ & $0.0347388949$ & $0.5250986933$ & $0.1614803309$\\
\bottomrule
\end{tabular}
\end{table}

A direct fourth-order Runge--Kutta solver, using the modified agent
accelerations in \eqref{eq:self-weight-model}, a maximum step
$2\times10^{-4}$, and the same event-bracketing tolerance as above,
reproduces all four deletions for each self-weight. The largest absolute
event-time discrepancy is below $7\times10^{-13}$. Comparing the stationary
projection of the directly computed terminal-mode state with the predicted
limiting velocity gives discrepancies below $1.4\times10^{-13}$. This
projection comparison avoids identifying a finite-time velocity with its
limit. These discrepancies describe the observed runs, not rigorous error
bounds for the numerical method.

For the non-monotone data in Example~\ref{ex:safe-entry}, the explicit
certificate gives $T_{\rm ord}=0.3799100105$ and
$D_0T_{\rm ord}=0.2279460063<0.30$. The predicted first fragmentation time
is $0.8863048415$; direct integration agrees within $3\times10^{-13}$.
The certificate excludes any preceding topology change or crossing, so
this comparison uses the actual path trajectory through entry.
The script \texttt{validate\_\allowbreak extensions.py} and its CSV and JSON outputs
record these additional checks, including the complete event sequences. The
reproduction scripts and generated outputs are supplied with the manuscript
as Online Resource~1.

\FloatBarrier
\section{Discussion}
\label{sec:discussion}

At the finite-particle level, a hard interaction cut-off creates a selection
problem between global flocking and fragmentation. The results above solve
that problem exactly in the expansive regime. Monotone separation converts
the total displacement remaining under velocity relaxation into a decision
about whether an active interaction survives. The group inverse measures
that displacement, while the interaction slack measures the distance to
loss of contact. Their comparison determines the next event and, after
iteration, the complete terminal state. The equality case is physically and
mathematically distinct: the particles approach the cut-off asymptotically
without losing the interaction at finite time.

For path configurations, the Green matrix makes the cluster-selection
mechanism explicit. Every initial velocity gradient contributes to the
expansion of every active gap, and the critical alignment rate is the largest
of the corresponding nonlocal gap ratios. Thus the threshold is a collective
finite-size quantity rather than a condition attached to one pair of
particles. The recursion retains this collective character on non-path
graphs because every deletion changes the local normalization and hence the
remaining displacement of all surviving interactions.

Theorem~\ref{thm:monotone-relaxation} identifies the structural ingredients
behind the calculation: an invariant cone, monotone switching observables,
and fixed-mode relaxation with no limiting drift in active relative
coordinates. Proposition~\ref{prop:self-weight} shows that the mechanism is
not tied to the self-excluding convention; it persists for every common
self-weight $\eta\ge0$. The spectral entry result further gives an open set
of initially nonordered data for which the exact theory becomes applicable
before any geometric event.

Several natural extensions require different ideas. With arbitrary mixed
velocities, a separation can overshoot the cut-off and later contract, and
new interactions can be created. Distance-dependent weights vary even while
the active edge set is unchanged, so the constant-matrix group-inverse
formula no longer gives the exact accumulated motion. In higher dimensions
there is no total spatial order that makes all relevant distances monotone.
Convergence theory alone \cite{HendrickxTsitsiklis2013} does not recover the
missing event sequence in these settings.

The analysis is deliberately finite-size. Establishing a kinetic or
mean-field limit would require uniform control of the hard cut-off and of the
normalization near regions of small local mass; existing particle-to-
continuum results for other alignment systems \cite{HaTadmor2008,HaLiu2009}
do not directly provide that control. The exact trajectories, thresholds,
and terminal configurations obtained here supply finite-particle benchmarks
for such a theory. The numerical comparisons verify the implementation on
paths, a non-path interval graph, and a safely entering configuration, while
the proofs apply to the full classes stated in the theorems.

\section*{Statements and Declarations}

\noindent\textbf{Funding.}
The author received no financial support for the research, authorship, or
publication of this article.

\medskip
\noindent\textbf{Competing interests.}
The author has no relevant financial or non-financial interests to disclose.

\medskip
\noindent\textbf{Data availability.}
All numerical data generated for this study and the scripts used to reproduce
the computations are provided with the manuscript as Online Resource~1.

\bibliographystyle{unsrturl}
\bibliography{fragmentation_alignment_JSP}

\end{document}